\documentclass[psamsfonts]{amsart}

\usepackage[a4paper,bindingoffset=0.0in,left=.975in,right=.975in,top=1.50in,bottom=1.50in,footskip=.35in]{geometry}

\usepackage{amssymb,amsfonts}
\usepackage[all,arc]{xy}
\usepackage{enumerate}
\usepackage{mathrsfs}
\usepackage{amsmath}
\usepackage{setspace}
\usepackage{tikz-cd}

\newtheorem{thm}{Theorem}[section]
\newtheorem{cor}[thm]{Corollary}
\newtheorem{prop}[thm]{Proposition}
\newtheorem{lem}[thm]{Lemma}

\theoremstyle{definition}
\newtheorem{defn}[thm]{Definition}

\newtheorem{exmp}[thm]{Example}
\newtheorem{exmps}[thm]{Examples}

\theoremstyle{remark}
\newtheorem{rem}[thm]{Remark}
\newtheorem{rems}[thm]{Remarks}

\newcommand{\C}{\Bbb{C}}

\newcommand{\Q}{\Bbb{Q}}

\newcommand{\II}{\mathcal{I}}

\newcommand{\HH}{\mathcal{H}}

\newcommand{\F}{\Bbb{F}}

\newcommand{\Z}{\Bbb{Z}}

\newcommand{\m}{\frak{m}}

\newcommand{\Mod}{\text{Mod}}

\newcommand{\bb}{\bullet}

\newcommand{\ind}{\text{ind}}

\newcommand{\Hom}{\text{Hom}}

\newcommand{\Ext}{\text{Ext}}

\newcommand{\End}{\text{End}}

\newcommand{\im}{\text{im}}

\newcommand{\GL}{\text{GL}}

\newcommand{\SL}{\text{SL}}

\newcommand{\op}{\text{op}}

\newenvironment{psmallmatrix}
  {\left(\begin{smallmatrix}}
  {\end{smallmatrix}\right)}
  
\makeatletter
\let\c@equation\c@thm
\makeatother
\numberwithin{equation}{section}

\title{On the graded center of $D(G)^c$}

\author{Peter Schneider and Claus Sorensen}

\date{\today}

\begin{document}

\begin{abstract}
Let $D(G)$ denote the derived category of smooth $G$-representations on $k$-vector spaces, where $G$ is a locally pro-$p$ group
and $k$ is a field of characteristic $p$. In this paper we are primarily interested in the graded center of the subcategory of compact objects 
$Z^*(D(G)^c)$ and variants thereof. When $G$ is a $p$-adic Lie group, without proper open centralizers, we completely determine this center modulo locally nilpotent elements and give various applications.
\end{abstract}

\maketitle



\onehalfspacing

\section{Introduction}

For a $p$-adic reductive group $G$, the Bernstein decomposition makes precise how the category of smooth complex representations $\Mod_{\C}(G)$ is built from supercuspidal representations of Levi subgroups via parabolic induction. This decomposition is optimal in the sense that each block is indecomposable. One way to see this is to compute the Bernstein center $Z(\Mod_{\C}(G))$ which factors accordingly. 

One of the players in the $p$-adic Langlands program is the category $\Mod_{k}(G)$ of smooth $G$-representations on $k$-vector spaces where $k$ is now a field of characteristic $p$. 
Since $k$ is fixed throughout this paper we will suppress it from the notation and just denote this category by $\Mod(G)$. Although $p$-adic reductive groups are perhaps our main interest, we are allowing $G$ to be any locally pro-$p$ group (i.e. $G$ has an open subgroup that is pro-$p$ in the subspace topology). Examples to keep in mind are 
$G={\bf{G}}(\frak{F})$ for a linear algebraic group ${\bf{G}}$ defined over \emph{any} non-archimedean local field $\frak{F}$; in other words a finite extension of $\Q_p$ or $\Bbb{F}_p(\!(t)\!)$.

There are always 'mundane' elements of $Z(\Mod(G))$ arising from the center of the group $Z(G)$. In fact any object of $\Mod(G)$ can be promoted to a module over the completed group algebra 
\begin{equation}\label{compl}
\widehat{k[Z(G)]}:={\varprojlim}_{Z' \subset Z(G)} k[Z(G)/Z']
\end{equation}
with $Z'$ running over the compact open subgroups of $Z(G)$,
and this gives an injective homomorphism of $k$-algebras
$$
\Phi: \widehat{k[Z(G)]} \longrightarrow Z(\Mod(G)). 
$$
This is not always an isomorphism (for instance if $G$ is a non-abelian finite group). However, it turns out it \emph{is} an isomorphism in most of the cases we are interested in, such as the aforementioned groups ${\bf{G}}(\frak{F})$ provided ${\bf{G}}$ is connected (and smooth if $\frak{F}$ has characteristic $p$). More generally \cite[Thm.~1.1]{AS23} asserts that $\Phi$ is an isomorphism if $G$ has 'no proper open centralizers'. As a first auxiliary result of this article, in Proposition \ref{addendum} we give an addendum to this Theorem (under the same hypotheses) by showing that $\Phi$ identifies the topological nilreduction $\widehat{k[Z(G)]}_\text{red}$ with $Z(\Mod(G)^\text{fg})$ modulo locally nilpotent elements. The subcategory of finitely generated representations need not be an abelian category (unless we are in the rare situation where $\Mod(G)$ is locally Noetherian) but it is nevertheless additive and it makes sense to talk about its center. 

In this article our principal goal is to better understand the graded center of the \emph{derived} category $D(G)$ of $\Mod(G)$ from \cite{Sch15} and variants thereof; particularly the subcategory $D(G)^c$ of compact objects. 



We now expand on the main results of our article. For clarity let us summarize our standing hypotheses.
\begin{itemize}
\item $k$ is a field of characteristic $p$;
\item $G$ is a locally pro-$p$ group without proper open centralizers. 
\end{itemize}
The centralizer assumption is rather mild, see \cite[Thm.~6.13]{AS23}. The basic idea to, at the minimum, get our hands on the degree \emph{zero} component $Z^0(D(G))$ is to map it to the Bernstein center 
via the natural restriction map 
$$
\pi: Z^0(D(G)) \longrightarrow Z(\Mod(G))
$$
arising from viewing $\Mod(G)$ as a full subcategory of $D(G)$. There is also a canonical map $s$ going in the other direction, but this is only a right-inverse
($\pi \circ s=\text{Id}$) in general. If $\mathcal{S}\subset D(G)$ is a triangulated subcategory, the subtlety lies in finding a corresponding subcategory
$\Mod(G)_\mathcal{S} \subset \Mod(G) \cap \mathcal{S}$ for which one can (a) control the kernel of $Z^0(\mathcal{S})\overset{\pi_\mathcal{S}}{\rightarrow} Z(\Mod(G)_\mathcal{S})$ -- and 
(b) determine the image of $\pi_\mathcal{S}$.

In the next assertion we gather Corollary \ref{maincor}, Theorem \ref{bounded}, and Corollary \ref{loccoh} in one place. 

\begin{thm}\label{introthm}
Under the above hypotheses the map $Z(G) \ni c \mapsto (s\circ\Phi)(c)$ induces the following isomorphisms of $k$-algebras.
\begin{itemize}
\item[(1)] Under the additional assumption that $G$ is a $p$-adic Lie group we have the isomorphism 
$$
\widehat{k[Z(G)]}_\text{red} \overset{\sim}{\longrightarrow} Z^0(D(G)^c)_\text{lred}.
$$ 
\item[(2)] When the category $\Mod(G)$ is locally coherent we have the isomorphism
$$
\widehat{k[Z(G)]}_\text{red} \overset{\sim}{\longrightarrow} Z^0(D_\text{fp}^b(G))_\text{lred}. 
$$
\item[(3)] Without further conditions on $G$ we have the isomorphism $\widehat{k[Z(G)]}_\text{nil} \overset{\sim}{\longrightarrow} Z^0(D^b(G))_\text{lred}$.
\end{itemize}
\end{thm}

Let us briefly explain the meaning of the various subscripts in the statement. Writing 'red' means we mod out by the ideal of topologically nilpotent elements, and 'nil' means modding out the actual nilpotent elements. Finally 'lred' on the right-hand side indicates we divide
by the ideal of locally nilpotent elements. We refer to the main text for any unexplained notation. 

Theorem \ref{bounded} also contains partial results on the graded center of $D^+(G)$ and the whole category $D(G)$. As a matter of fact, we go to great lengths to prove that the kernel
$$
\ker\big(Z^*(D(G))\overset{\text{res}}{\longrightarrow} Z^*(D^+(G))\big)
$$
has square zero provided $G$ is a $p$-adic Lie group. This involves establishing several results, for which we have not been able to find a suitable reference in the literature, in the generality of Grothendieck categories satisfying Roos's condition AB4*-$d$ (i.e. taking products has finite cohomological amplitude). Hopefully this will also prove useful to researchers in other areas. 

In Section \ref{appblock} we apply part (1) of Theorem \ref{introthm} to determine the decomposition of $D(G)$ into indecomposable blocks. The corresponding result for $\Mod(G)$ follows straightforwardly from \cite{AS23} and we expose the details. Besides our standing hypotheses we assume that $Z(G)$ is topologically finitely generated (as is the case for $p$-adic reductive groups for example). In particular $Z(G)$ has a unique maximal compact subgroup, which we factor as $P\times Q$ where $P$ is the Sylow pro-$p$ subgroup and $Q$ is a finite group of order prime-to-$p$. Thus $k[Q]\simeq\prod_{i \in I} k_i$ is a product of finitely many finite field extensions $k_i/k$ and we let $(e_i)_{i \in I}$ denote the collection of primitive idempotents in $k[Q]$. The analogue of the Bernstein decomposition is 
$$
\Mod(G)=\prod_{i \in I} \Mod^{e_i}(G)
$$ 
where the block $\Mod^{e_i}(G)$ consists of representations on which $e_i$ acts as the identity. The key point here is that $\Mod^{e_i}(G)$
is indecomposable. In other words the factorization of $\Mod(G)$ mirrors that of $k[Q]$. Informally this is saying there is no 'unapparent' way to decompose $\Mod(G)$ into smaller pieces. 

If $G$ is a $p$-adic Lie group, still with a topologically finitely generated center, we derive a parallel result for $D(G)$. The indecomposable blocks are 
$D^{e_i}(G):=D(\Mod^{e_i}(G))$. Again, we emphasize that the main point here is the indecomposability, which we prove by passing to subcategories of compact objects and employing our $Z^0(D(G)^c)_\text{lred}$ calculation. As an example, $D(G)$ itself is indecomposable if $Z(G)$ is trivial.  

In Section \ref{hoch} we fix a pro-$p$ open subgroup $U\subset G$ (for a general $G$ as above) and consider the Yoneda algebra 
$\Ext_{\Mod(G)}^*(\Bbb{X}_U, \Bbb{X}_U)^\text{op}$ where $\Bbb{X}_U:=k[G/U]$. The degree zero part is the Hecke algebra $\mathcal{H}_U$. We are interested in the degree zero part of the graded center $Z^0(\Ext_{\Mod(G)}^*(\Bbb{X}_U, \Bbb{X}_U))$. By definition this is the subalgebra of $Z(\mathcal{H}_U)$ comprising the elements that commute with every element of the full $\Ext$-algebra (in all degrees) and it is typically much smaller than $Z(\mathcal{H}_U)$. There is an evident way to at least produce a portion of the center of interest,
namely $Z(G) \ni c \mapsto \tau_c:=\text{char}_{UcU}$ extends to an injective homomorphism 
$$
k[Z(G)/Z(G)\cap U]\hookrightarrow Z^0(\Ext_{\Mod(G)}^*(\Bbb{X}_U, \Bbb{X}_U)).
$$
For split $p$-adic reductive groups, and with $U$ being a pro-$p$ Iwahori subgroup, this was shown to also be surjective in Bodon's M\"{u}nster PhD Thesis \cite{Bod22}.
By contrast, in general there are constraints on how large the image can be: As we observe in (the proof of) Theorem \ref{imchar} each $\tau_c$ lies in the image of the characteristic map 
$$
\chi_U: HH^0(\HH_U^\bb) \longrightarrow Z^0(\Ext_{\Mod(G)}^*(\Bbb{X}_U, \Bbb{X}_U)).
$$
Here $\HH_U^\bb$ is a differential graded Hecke algebra, as defined in \cite{Sch15}, and $HH^0$ is Hochschild cohomology in degree zero. Additionally we prove that the resulting map $k[Z(G)/Z(G)\cap U] \rightarrow \im(\chi_U)$ becomes an isomorphism after quotienting out the nilpotents. In particular this gives a negative answer to Question 4.4 in \cite{Har16}; 
the canonical map $HH^0(\HH_U^\bb) \rightarrow Z(\mathcal{H}_U)$ is not always surjective.


\section{Some ring-theoretic preliminary remarks}

The center $Z(R)$ of a ring $R$ is isomorphic to the center of the category of $R$-modules. In the derived setting the situation is slightly more complicated. In the derived category of $R$-modules $D(R)$ the subcategory of compact objects $D(R)^c$ is given by the perfect complexes. There is a natural way to associate a graded center with any triangulated category, and the degree zero piece $Z^0(D(R)^c)$ consists of the natural transformations $t: \text{Id}\rightarrow \text{Id}$ which commute with the suspension functor $\Sigma$. There is now a canonical injective homomorphism $Z(R)\overset{\frak{c}}{\longrightarrow} Z^0(D(R)^c)$ with a left inverse given by evaluation at $R$. However, this is not always an isomorphism. For instance, $\frak{c}$ is not surjective even for the dual numbers over a field $R=k[x]/(x^2)$ as \cite[Rk.~4.15]{Rou10} shows. On the other hand, by \cite[Prop.~4.13]{Rou10} $\frak{c}$ is an isomorphism when $R$ is without zero-divisors. For a general ring $R$ (the proof of) \cite[Lem.~4.10]{Rou10} shows that $\frak{c}$ induces an isomorphism $Z(R)_\text{red} \overset{\sim}{\longrightarrow} Z^0(D(R)^c)_\text{lred}$. Here the subscript 'lred' means we divide out by the ideal of locally nilpotent elements, i.e. the natural transformations $t$ as above inducing a nilpotent endomorphism of every object in $D(R)^c$.

Instead of a ring we can consider a differential graded algebra $R^\bb$. For the purpose of these preliminary remarks we deem its center to be the zeroth Hochschild cohomology 
$HH^0(R^\bb)$. There is an analogous homomorphism $HH^0(R^\bb)\overset{\frak{c}}{\longrightarrow} Z^0(D(R^\bb)^c)$ and we like to think of the target as a 'good approximation' to the center of $R^\bb$. In contrast to the Hochschild cohomology, $Z^0(D(R^\bb)^c)$ is actually often computable (at least up to locally nilpotents). One of the main goals of this article 
is the computation of this ring when $R^\bb$ is a differential graded Hecke algebra attached to a pair $(G,U)$ comprising a locally profinite group $G$ and a compact open subgroup $U$ thereof. In fact we prefer to stay on the representation-theoretic side of the picture and directly (and equivalently) compute $Z^0(-)_\text{lred}$ of the thick envelope $\text{Thick}_{D(G)}(\text{ind}_U^G(1))$ without invoking the dg Hecke algebra. In the case where $G$ is a $p$-adic Lie group, and our suppressed coefficient field has characteristic $p$, this thick envelope is precisely $D(G)^c$ provided $U$ is a torsion-free pro-$p$ open subgroup. This is our motivating example. 

\section{General results on graded centers for locally pro-$p$ groups}

\subsection{General results about the center of a Grothendieck category}

The next result is an outgrowth of \cite[Appendix]{AS23}. 

\begin{lem}\label{groth}
Let $\mathcal{C}$ be a Grothendieck category, and let $\mathcal{A}$ be a full additive subcategory containing a family of generators for $\mathcal{C}$. Then the restriction map
$Z(\mathcal{C})\rightarrow Z(\mathcal{A})$ is surjective. 
\end{lem}

\begin{proof}
We give ourselves a tuple $(t_V)_{V \in \mathcal{A}}$ in $Z(\mathcal{A})$ and our goal is to extend it compatibly to all objects of the ambient category $\mathcal{C}$. 

As an intermediate step we first extend the tuple to objects which are coproducts of objects from $\mathcal{A}$. Thus, suppose $V$ is a coproduct 
$\bigoplus_{i \in I}V_i$ of objects $V_i \in \mathcal{A}$, with coprojections $\sigma_i : V_i \rightarrow V$. We define $t_V$ to be the unique morphism $t_V: V\rightarrow V$ 
such that $t_V \circ \sigma_i=\sigma_i \circ t_{V_i}$ for all $i$. We must verify that $t_V$ is independent of choices and natural in $V$. To see this, suppose $V'$ is a coproduct 
$\bigoplus_{j \in J}V_j'$ with $V_j' \in \mathcal{A}$ and let $\sigma_j': V_j' \rightarrow V'$ denote the coprojections. As before, $t_{V'}: V' \rightarrow V'$ is the morphism with the property that $t_{V'}\circ \sigma_j'=\sigma_j' \circ t_{V_j'}$ for all $j$. Given a morphism $\gamma: V \rightarrow V'$ our task is to check that $t_{V'}\circ \gamma=\gamma \circ t_V$. 
(Taking $\gamma$ to be the identity in particular shows $t_V$ is independent of how we exhibit $V$ as a coproduct.) Equivalently, we need to check that
$t_{V'}\circ \gamma \circ \sigma_i=\gamma \circ t_V \circ \sigma_i$ for all $i$. 

For the purpose of this proof, let $\pi_j': V' \rightarrow V_j'$ denote the morphism for which 
$$
\pi_j' \circ \sigma_{l}'=\begin{cases} \text{Id}_{V_j'} & j=l \\ 0 & j \neq l. \end{cases}
$$
Since $\mathcal{C}$ is Grothendieck, the resulting morphism $V'=\bigoplus_{j \in J}V_j'\rightarrow \prod_{j \in J}V_j'$ is a monomorphism (the AB5 condition allows us to reduce to the case of finite index sets); see \cite[Exc.~1, p.~133]{Ste75}. It therefore suffices to prove that
$$
\pi_j' \circ t_{V'}\circ \gamma \circ \sigma_i=\pi_j' \circ \gamma \circ t_V \circ \sigma_i
$$
for all $i$ and $j$. As a preliminary observation we note that $\pi_j' \circ t_{V'}=t_{V_j'}\circ \pi_j'$ by precomposing both sides with $\sigma_l'$. Now, 
by unwinding the definitions we see that indeed
$$
\pi_j' \circ t_{V'}\circ \gamma \circ \sigma_i=
t_{V_j'}\circ (\pi_j' \circ \gamma \circ \sigma_i)=(\pi_j' \circ \gamma \circ \sigma_i) \circ t_{V_i}=
\pi_j' \circ \gamma \circ t_V \circ \sigma_i.
$$
In the second step we used the fact that our given tuple is in $Z(\mathcal{A})$ and therefore $\pi_j' \circ \gamma \circ \sigma_i$ intertwines $t_{V_i}$ and 
$t_{V_j'}$. This concludes the proof of the intermediate step. 

Next we have to extend the tuple to all of $\mathcal{C}$. For convenience we let $\frak{C}$ denote the full subcategory of $\mathcal{C}$ whose objects are coproducts of objects from 
$\mathcal{A}$. We just extended our given tuple to an element  $(t_V)_{V \in \frak{C}}$ in $Z(\frak{C})$. By our hypothesis on $\mathcal{A}$, any $V \in \mathcal{C}$ admits an epimorphism $\beta: C \rightarrow V$ for some $C \in \frak{C}$; see \cite[Sect.~IV.6]{Ste75}. We choose a kernel $\kappa: \ker(\beta) \rightarrow C$ and an epimorphism
$\beta': C' \rightarrow \ker(\beta)$ with $C' \in \frak{C}$. By the naturality of $(t_V)_{V \in \frak{C}}$ the solid diagram 
$$
\begin{tikzcd}
C' \arrow[r, "\beta'"] \arrow[d, "t_{C'}"]
& \ker(\beta) \arrow[r, "\kappa"] \arrow[d, dashrightarrow]& C \arrow[d, "t_C"] \\
C' \arrow[r, "\beta'"]
& \ker(\beta) \arrow[r, "\kappa"] & C
\end{tikzcd}
$$
commutes. In particular $\beta \circ t_C \circ \kappa=0$ since $\beta'$ is an epimorphism and 
$$
\beta \circ t_C \circ \kappa\circ \beta'=\beta \circ \kappa \circ \beta' \circ t_{C'}= 0 \circ \beta' \circ t_{C'}=0.
$$
Therefore, by the universal property of $\kappa$, there is a unique dashed arrow as in the diagram which makes the right square commute. Also, since $\beta: C \rightarrow V$ is an epimorphism, $\beta$ is a cokernel of $\kappa$ (as $\mathcal{C}$ is abelian). We conclude from $\beta \circ t_C \circ \kappa=0$ that there is a unique morphism $t_V: V \rightarrow V$ such that $t_V \circ \beta=\beta\circ t_C$. 

To see that this $t_V$ is independent of the choice of $\beta$, suppose we are given two epimorphisms $\beta_i: C_i \rightarrow V$ ($i=1,2$) with $C_i \in \frak{C}$, and let 
$t_V^{(i)}: V \rightarrow V$ be the morphism satisfying the relation $t_V^{(i)} \circ \beta_i=\beta_i\circ t_{C_i}$ as in the previous paragraph. If there is a morphism
$\delta: C_1 \rightarrow C_2$ such that $\beta_1=\beta_2\circ \delta$, we infer that $t_V^{(1)}=t_V^{(2)}$ by precomposing both sides with the epimorphism $\beta_1$, i.e.
$$
t_V^{(1)} \circ \beta_1=\beta_1 \circ t_{C_1}=\beta_2\circ \delta \circ t_{C_1}=\beta_2 \circ t_{C_2}\circ \delta=t_V^{(2)} \circ \beta_2 \circ \delta=t_V^{(2)} \circ \beta_1.
$$
In general (when there is no such $\delta$) we reduce to this special case by considering the diagram
$$
\begin{tikzcd}
C_1 \arrow[rd, "\beta_1"'] \arrow[r, "\delta_1"] & C_1\oplus C_2 \arrow[d] & C_2 \arrow[l, "\delta_2"'] \arrow[ld, "\beta_2"]\\
& V
\end{tikzcd}
$$
which shows that both $t_V^{(1)}$ and $t_V^{(2)}$ agree with the morphism defined by middle epimorphism in the diagram. 

Finally we must establish that our extended tuple $(t_V)_{V \in \mathcal{C}}$ indeed is an element of $Z(\mathcal{C})$, i.e. that $t_{V_2}\circ \phi=\phi \circ t_{V_1}$ for all morphisms
$\phi: V_1 \rightarrow V_2$. We choose epimorphisms $\beta_i: C_i \rightarrow V_i$ ($i=1,2$) with $C_i \in \frak{C}$ as above (but now $\beta_i$ has a different meaning, of course). 
Then the defining property of $t_{V_i}$ is the identity $t_{V_i} \circ \beta_i=\beta_i\circ t_{C_i}$. However, by our previous remarks we may replace $\beta_2$ by 
$\tilde{\beta}_2:=\phi \circ \beta_1 +\beta_2$ in this identity, as in the following commutative diagram:
$$
\begin{tikzcd}
C_1 \arrow[r, "t_{C_1}"] \arrow[d, "\beta_1"]
& C_1 \arrow[d, "\beta_1"] \arrow[r, "\delta_1"] & C_1\oplus C_2 \arrow[r, "t_{C_1\oplus C_2}"]\arrow[d, "\tilde{\beta}_2"] & C_1\oplus C_2 \arrow[d, "\tilde{\beta}_2"]\\
V_1 \arrow[r, "t_{V_1}"] \arrow[r] & V_1 \arrow[r, "\phi"] & V_2 \arrow[r, "t_{V_2}"]
& V_2.
\end{tikzcd}
$$
Note that $t_{C_1\oplus C_2}\circ \delta_1=\delta_1 \circ t_{C_1}$, from which we deduce that $t_{V_2}\circ \phi\circ \beta_1=\phi \circ t_{V_1}\circ \beta_1$ by inspecting the diagram. 
Since $\beta_1$ is an epimorphism we conclude that indeed $t_{V_2}\circ \phi=\phi \circ t_{V_1}$ as desired. 
\end{proof}

We say that $(t_V)_{V \in \mathcal{C}}\in Z(\mathcal{C})$ is locally nilpotent if $t_V$ is a nilpotent element of $Z(\text{End}_{\mathcal{C}}(V))$ for all $V \in \mathcal{C}$. The locally nilpotent tuples clearly form an ideal of $Z(\mathcal{C})$ and we let $Z(\mathcal{C})_\text{lred}$ denote the quotient ring. Similarly for $\mathcal{A}$. We let $\text{Serre}_\mathcal{C}(\mathcal{A})$ denote the smallest Serre subcategory of $\mathcal{C}$ containing $\mathcal{A}$. 

\begin{lem}\label{serre}
In the setup of Lemma \ref{groth}, the restriction map induces an isomorphism 
$$
Z(\text{Serre}_\mathcal{C}(\mathcal{A}))_\text{lred} \overset{\sim}{\longrightarrow}Z(\mathcal{A})_\text{lred}.
$$
\end{lem}

\begin{proof}
The surjection $Z(\mathcal{C})\rightarrow Z(\mathcal{A})$ from Lemma \ref{groth} factors through $Z(\text{Serre}_\mathcal{C}(\mathcal{A})) \rightarrow Z(\mathcal{A})$ which is therefore also surjective. This obviously remains true after passing to lred. 

Given a $(t_V)_{V \in \text{Serre}_\mathcal{C}(\mathcal{A})}$ in the center of $\text{Serre}_\mathcal{C}(\mathcal{A})$ which is nilpotent on all objects from $\mathcal{A}$, we must show that $t_V$ is nilpotent for all $V \in \text{Serre}_\mathcal{C}(\mathcal{A})$. For that purpose, it suffices to prove that the auxiliary subcategory
$$
\mathcal{S}:=\{V \in \text{Serre}_\mathcal{C}(\mathcal{A}): \text{$t_V$ is nilpotent}\}
$$
is Serre in $\mathcal{C}$. So, let $0 \rightarrow V' \overset{\alpha}{\rightarrow} V \overset{\beta}{\rightarrow} V'' \rightarrow 0$ be a short exact sequence in $\mathcal{C}$. First, if $V \in \mathcal{S}$ then at least $V',V''\in \text{Serre}_\mathcal{C}(\mathcal{A})$, and from the commutative diagram
$$
\begin{tikzcd}
0 \arrow[r] & V' \arrow[r,"\alpha"] \arrow[d, "t_{V'}^n"]
& V \arrow[d, "t_V^n"] \arrow[r, "\beta"] & V'' \arrow[d, "t_{V''}^n"] \arrow[r] & 0\\
0 \arrow[r] & V' \arrow[r, "\alpha"]
& V \arrow[r, "\beta"] & V'' \arrow[r] & 0
\end{tikzcd}
$$
it is clear that if $t_V^n=0$, for some $n$ depending on $V$, then both $t_{V'}^n=0$ and $t_{V''}^n=0$ (since $\alpha$ is a monomorphism and $\beta$ is an epimorphism, respectively). Thus in fact $V', V'' \in \mathcal{S}$. 

Conversely, now we assume that $V', V'' \in \mathcal{S}$ (and deduce that $V \in \mathcal{S}$). Once and for all, choose an $n$ such that $t_{V'}^n=0$ and $t_{V''}^n=0$. Note that at least $V \in \text{Serre}_\mathcal{C}(\mathcal{A})$ so $t_V$ is defined. Our goal is to show that $t_V^{2n}=0$. From the previous diagram one sees that $t_V^n$ factors through 
$V \rightarrow \text{cok}(\alpha)$ and also through $\ker(\beta)\rightarrow V$. This yields the diagram
$$
\begin{tikzcd}
& \text{cok}(\alpha) \arrow[dr, dashrightarrow] & \\
V \arrow[r, "t_V^n"] \arrow[dr, dashrightarrow]& V \arrow[r, "t_V^n"] \arrow[u] & V \\
& \ker(\beta) \arrow[u]& 
\end{tikzcd}
$$
in which the vertical composition is zero by the very definition of exactness. Hence $t_V^{2n}=t_V^n \circ t_V^n=0$. 
\end{proof}

In particular, if $\mathcal{A}_1, \mathcal{A}_2 \subset \mathcal{C}$ are two subcategories both satisfying the hypotheses of Lemma \ref{groth}, and we assume they generate the same Serre subcategory of $\mathcal{C}$, then there are isomorphisms
$$
Z(\mathcal{A}_1)_\text{lred} \overset{\sim}{\longleftarrow} Z(\text{Serre}_\mathcal{C}(\mathcal{A}_1))_\text{lred}=
Z(\text{Serre}_\mathcal{C}(\mathcal{A}_2))_\text{lred} \overset{\sim}{\longrightarrow}Z(\mathcal{A}_2)_\text{lred}.
$$


\subsection{On the center of the category of finitely generated representations}

We adopt the setup in \cite{AS23}. Thus $G$ denotes a locally profinite group which admits a pro-$p$ open subgroup, $k$ is a field of characteristic $p$, and $\Mod(G)$ is the category of smooth $G$-representations on $k$-vector spaces. Usually $U$ will denote a compact open subgroup of $G$, in which case 
$\Bbb{X}_U=\text{ind}_U^G(1)=k[G/U]$ denotes the compactly induced representation.

\begin{lem}
For a fixed compact open subgroup $U \subset G$, the set of representations $\{\Bbb{X}_{U'}\}_{U'\subset U}$ is a family of generators for $\Mod(G)$.
\end{lem}

\begin{proof}
Suppose $V \in \Mod(G)$ is the $k$-span of $\{v_i: i \in I\}$. For each $i$ pick an open subgroup $U_i \subset U$ such that $v_i \in V^{U_i}$. Frobenius reciprocity gives a morphism
$\Bbb{X}_{U_i}\rightarrow V$ sending $\text{char}_{U_i}$ to $v_i$. Their sum is a surjective morphism $\bigoplus_{i \in I} \Bbb{X}_{U_i}\twoheadrightarrow V$.
\end{proof}

$\Mod(G)$ is a Grothendieck category and we now consider the Serre subcategory generated by a single representation $\Bbb{X}_U$, with $U$ pro-$p$. It turns out this subcategory is independent of the choice of $U$, and it coincides with the Serre subcategory generated by the finitely generated objects $\Mod(G)^\text{fg}$ (which of course itself is Serre if $\Mod(G)$ is locally noetherian, see \cite[Prop.~4.2, p.~123]{Ste75}). In general $\Mod(G)^\text{fg}$ is just an additive subcategory of $\Mod(G)$. Our immediate goal is to compute its center up to locally nilpotent elements; see Proposition \ref{addendum} below.

\begin{lem}\label{fg}
Let $U \subset G$ be a pro-$p$ open subgroup. Then $\text{Serre}_{\Mod(G)}(\Bbb{X}_U)=\text{Serre}_{\Mod(G)}(\Mod(G)^\text{fg})$.
\end{lem}

\begin{proof}
The inclusion $\subseteq$ is obvious since $\Bbb{X}_U$ is generated by $\text{char}_U$. For the opposite inclusion $\supseteq$ we write a finitely generated $V$ as a quotient 
of a finite coproduct $\bigoplus_{i \in I} \Bbb{X}_{U_i}\twoheadrightarrow V$ with $U_i \subset U$. We are reduced to verifying that 
$\Bbb{X}_{U'}\in \text{Serre}_{\Mod(G)}(\Bbb{X}_U)$ for all open subgroups $U' \subset U$. In fact we will check that indeed $\text{Serre}_{\Mod(G)}(\Bbb{X}_U)$ contains 
$\text{ind}_U^G(F)$ for all finite-dimensional smooth $k[U]$-modules $F$. (Taking $F=\text{ind}_{U'}^U(1)$ then gives the result.) We use induction of $\dim F$. The base case is
trivial, so suppose $\dim F>0$ and assume the claim holds in smaller dimensions. Note that $\dim F^U>0$ since $U$ is pro-$p$. In the short exact sequence
$$
0 \longrightarrow \text{ind}_U^G(F^U) \longrightarrow \text{ind}_U^G(F) \longrightarrow \text{ind}_U^G(F/F^U) \longrightarrow 0
$$
the last term therefore lies in $\text{Serre}_{\Mod(G)}(\Bbb{X}_U)$ by the induction hypothesis, and the first term is a finite coproduct $\Bbb{X}_U^{\oplus \dim F^U}$. We conclude that
the extension $\text{ind}_U^G(F)$ also belongs to $\text{Serre}_{\Mod(G)}(\Bbb{X}_U)$.
\end{proof}

We observe that $\text{Serre}_{\Mod(G)}(\Bbb{X}_U)$ hence contains the family of generators $\{\Bbb{X}_{U'}\}_{U'\subset U}$, and by Lemma \ref{groth} we infer that the restriction map 
$$
\rho: Z(\Mod(G)) \longrightarrow Z(\text{Serre}_{\Mod(G)}(\Bbb{X}_U))
$$
is surjective. In \cite[Lem.~3.7]{AS23} one finds a $k$-algebra homomorphism $\Phi: \widehat{k[Z(G)]}\rightarrow Z(\Mod(G))$. The main result \cite[Thm.~1.1]{AS23} asserts that
$\Phi$ is an isomorphism if $G$ has no proper open centralizers. (Here the hat denotes the completion of the group algebra $k[Z(G)]$ in the sense of \cite[Def.~3.3]{AS23}.)

\begin{lem}\label{inject}
Let $U \subset G$ be a pro-$p$ open subgroup. Then $\rho \circ \Phi$ factors as in the diagram
$$
\begin{tikzcd}
Z(\Mod(G)) \arrow[r, twoheadrightarrow, "\rho"] 
& Z(\text{Serre}_{\Mod(G)}(\Bbb{X}_U))_\text{lred}  \\
k[Z(G)] \arrow[u, "\Phi"] \arrow[r, twoheadrightarrow, "\text{pr}"]
& k[Z(G)/Z(G) \cap U]. \arrow[u, dashrightarrow, "\Psi"] 
\end{tikzcd}
$$
Moreover, the resulting map 
$$
\Psi: k[Z(G)/Z(G) \cap U]_\text{red} \rightarrow  Z(\text{Serre}_{\Mod(G)}(\Bbb{X}_U))_\text{lred}
$$
is injective. 
\end{lem}

\begin{proof}
The ideal $\ker(\text{pr})$ is generated by elements of the form $u-1$ with $u \in Z(G)\cap U$. Clearly any such $u$ acts as the identity on $\Bbb{X}_U$, so $\ker(\text{pr})$ annihilates 
$\Bbb{X}_U$. Now consider the auxiliary full subcategory
$$
\mathcal{D}:=\{V \in \Mod(G): \text{$\Phi_V(x)$ is nilpotent for all $x \in \ker(\text{pr})$}\}.
$$
Obviously $\mathcal{D}$ is closed under subs, quotients, and it contains $\Bbb{X}_U$. It is also closed under extensions, for suppose $0\rightarrow V' \rightarrow V \rightarrow V'' \rightarrow 0$ is a short exact sequence in $\Mod(G)$ with $V',V'' \in \mathcal{D}$ and $x \in \ker(\text{pr})$. Choose a large enough $n$ such that both 
$\Phi_{V'}(x)^n=0$ and $\Phi_{V''}(x)^n=0$. Then $\Phi_{V}(x)^{2n}=0$, as in the proof of Lemma \ref{serre}. We conclude that $\text{Serre}_{\Mod(G)}(\Bbb{X}_U) \subseteq \mathcal{D}$, which is to say that $(\rho \circ \Phi)(x)$ is locally nilpotent for all $x \in \ker(\text{pr})$.

The same argument shows that the morphism labeled $\Psi$ in the diagram actually factors through the nilreduction
$$
k[Z(G)]/\sqrt{\ker(\text{pr})} \overset{\sim}{\longrightarrow} k[Z(G)/Z(G) \cap U]_\text{red}.
$$
To see the resulting map is injective, we argue as follows. In the beginning, suppose $x=\sum_{c\in Z(G)}a_cc \in k[Z(G)]$ and $\Phi_{\Bbb{X}_U}(x)=0$, which means $xf=0$ for all $f \in \Bbb{X}_U$. Specializing to $f=\text{char}_U$, this amounts to 
$\sum_{c\in Z(G)}a_c \text{char}_{cU}=0$, which we break up as
$$
{\sum}_{r\in Z(G)/Z(G) \cap U} \bigg( {\sum}_{u\in Z(G) \cap U} a_{ru} \bigg) \text{char}_{rU}=0.
$$
We conclude that ${\sum}_{u\in Z(G) \cap U} a_{ru}=0$ for all $r$. By the same reasoning $\text{pr}(x)=\sum_r \bigg(\sum_u a_{ru}\bigg) r$, which is therefore zero. 
To summarize, this shows that if $\Phi(x)$ vanishes on $\Bbb{X}_U$ then $x \in \ker(\text{pr})$. So, if $\Phi(x)$ is nilpotent on $\Bbb{X}_U$ then $x \in \sqrt{\ker(\text{pr})}$.
In particular, if $(\rho \circ \Phi)(x)$ is locally nilpotent then $x \in \sqrt{\ker(\text{pr})}$.
\end{proof}

We wish to emphasize that the algebra $k[Z(G)/Z(G) \cap U]_\text{red}$ in Lemma \ref{inject} is independent of $U$. To be more precise, it is isomorphic to the topological nilreduction of $\widehat{k[Z(G)]}$, by which we mean the quotient $\widehat{k[Z(G)]}_\text{red}$ by the ideal of topologically nilpotent elements (i.e. those $x$ for which $x^i\rightarrow 0$ as $i \rightarrow \infty$). 

\begin{lem}\label{indepz}
Let $Z$ be an an arbitrary abelian locally pro-$p$ group and define $\widehat{k[Z]}$ by analogy with (\ref{compl}).
Then, for any pro-$p$ open subgroup $Z'\subset Z$, the natural projection  $\widehat{k[Z]}\twoheadrightarrow k[Z/Z']$ induces an isomorphism 
$$
\widehat{k[Z]}_\text{red}\overset{\sim}{\longrightarrow} k[Z/Z']_\text{red}.
$$
\end{lem}

\begin{proof}
If $Z''\subset Z'$ it will suffice to show that the kernel of the transition map $k[Z/Z''] \twoheadrightarrow k[Z/Z']$ consists of nilpotent elements, and it therefore induces an isomorphism when passing to nilreductions. To see this, note that this kernel is generated by elements $z-1$ with $z\in Z'/Z''$. If $|Z'/Z''|=p^N$ we find that indeed $(z-1)^{p^N}=z^{p^N}-1=0$.

Now, if $x \in \widehat{k[Z]}$ has nilpotent projection $x_{Z'}\in k[Z/Z']$, then $x$ must be topologically nilpotent. (Given an arbitrarily small $Z''\subset Z'$, some power of $x$ lies in the kernel of $\widehat{k[Z]} \twoheadrightarrow k[Z/Z'']$ as $x_{Z''}$ is nilpotent.)
\end{proof}

In our applications $Z=Z(G)$ and $Z'=Z(G)\cap U$ for a choice of pro-$p$ open subgroup $U \subset G$. In this setup
the natural projection therefore induces an isomorphism
$$
\widehat{k[Z(G)]}_\text{red} \overset{\sim}{\longrightarrow} k[Z(G)/Z(G)\cap U]_\text{red}.
$$

Let us now return to Lemma \ref{inject} and address the surjectivity of $\Psi$. We arrive at the following 'addendum' to \cite[Thm.~1.1]{AS23}.

\begin{prop}\label{addendum}
Let $G$ be a locally profinite group that contains a pro-$p$ open subgroup, and let $k$ be a field of characteristic $p$. Suppose $G$ contains no proper open centralizers. Then the natural map 
$$
\Psi: k[Z(G)/Z(G)\cap U]_\text{red} \overset{\sim}{\longrightarrow} Z(\text{Serre}_{\Mod(G)}(\Bbb{X}_U))_\text{lred}
$$
is an isomorphism for every pro-$p$ open subgroup $U \subset G$.
\end{prop}

\begin{proof}
Given a $t \in Z(\text{Serre}_{\Mod(G)}(\Bbb{X}_U))_\text{lred}$, there is at least an $x \in \widehat{k[Z(G)]}$ with $(\rho \circ \Phi)(x)=t$ since both maps $\rho$ and $\Phi$ are 
surjective by Lemma \ref{groth} and \cite[Thm.~1.1]{AS23}. We are free to replace $x$ by 
$x+y$ for any $y \in \ker(\widehat{\text{pr}})$, where $\widehat{\text{pr}}: \widehat{k[Z(G)]}\rightarrow k[Z(G)/Z(G) \cap U]$ temporarily denotes the natural projection defined on the completed group algebra. Since $k[Z(G)]$ is a dense subset of $\widehat{k[Z(G)]}$ and $\ker(\widehat{\text{pr}})$ is an open ideal hereof, we can arrange for $x$ to lie in 
$k[Z(G)]$ rather than just in its completion. Upon doing that we have the formula $(\Psi \circ \text{pr})(x)=t$, so that $\Psi$ maps $k[Z(G)/Z(G)\cap U]$ onto $Z(\text{Serre}_{\Mod(G)}(\Bbb{X}_U))_\text{lred}$.
\end{proof}

In light of Lemma \ref{serre} and Lemma \ref{fg}, Proposition \ref{addendum} identifies $Z(\Mod(G)^\text{fg})_\text{lred}$ with the algebra 
$\widehat{k[Z(G)]}_\text{red}$ we introduced above (under the 'no proper open centralizers' hypothesis of course). 

\subsection{Graded centers}

If $\mathcal{T}$ is a triangulated category, the graded center $Z^*(\mathcal{T})=\bigoplus_{r\in \Z} Z^r(\mathcal{T})$ is defined as follows. The elements of 
$Z^r(\mathcal{T})$ are natural transformations $t: \text{Id}\rightarrow \Sigma^r$ such that $t\Sigma=(-1)^r\Sigma t$. More concretely, $t$ is a functorial tuple 
$(t_V)_{V \in \mathcal{T}}$ of morphisms $t_V: V \rightarrow \Sigma^rV$ for which $t_{\Sigma V}=(-1)^r\Sigma(t_V)$. For example, in degree zero $Z^0(\mathcal{T})$ is the subalgebra of the full center of $\mathcal{T}$ consisting of all $t: \text{Id}\rightarrow \text{Id}$ commuting with $\Sigma$. There is an obvious composition product on $Z^*(\mathcal{T})$ which makes it a graded-commutative algebra, and the evaluation map $t \mapsto t_V$ gives a homomorphism $Z^*(\mathcal{T})\rightarrow \bigoplus_{r\in \Z} \Hom_{\mathcal{T}}(V,\Sigma^rV)$ into the graded center of the Yoneda 'Ext-algebra' of $V$.

If $\mathcal{X}$ is a set of objects in $\mathcal{T}$ we let $\text{Thick}_\mathcal{T}(\mathcal{X})$ denote the smallest thick subcategory of $\mathcal{T}$ containing $\mathcal{X}$. 
(By a 'thick' subcategory of $\mathcal{T}$ we mean a strictly full triangulated subcategory closed under direct summands.) 
We build it successively as in \cite[Sect.~2.2]{BvdB} and we briefly review the construction.

The starting point is the full subcategory $\langle \mathcal{X} \rangle_1=\text{smd}(\text{add}(\mathcal{X}))$. Here $\text{add}(\mathcal{X})$ denotes the smallest strictly full subcategory of $\mathcal{T}$, containing $\mathcal{X}$, which is closed under $\Sigma$ and finite coproducts. The notation $\text{smd}(-)$ means we are adjoining all direct summands. Then for $i\geq 2$ the increasing family of subcategories $\langle \mathcal{X}\rangle_i$ are defined inductively as
$$
\langle \mathcal{X}\rangle_i=\text{smd}(\langle \mathcal{X}\rangle_{i-1}\star \langle \mathcal{X}\rangle_1)=
\text{smd}(\underbrace{\langle \mathcal{X}\rangle_1 \star \cdots \star \langle \mathcal{X}\rangle_1}_\text{$i$ factors}).
$$
Here $\star$ is the operation introduced in \cite[Sect.~1.3.9, p.~33]{BBD}. For two sets of objects $\mathcal{A}$ and $\mathcal{B}$, we let 
$\mathcal{A} \star \mathcal{B}$ denote the set of objects $X$ sitting in a triangle $A\rightarrow X \rightarrow B \rightarrow$ with $A \in \mathcal{A}$ and $B \in \mathcal{B}$. 
This operation is associative by the octahedral axiom, see \cite[Lem.~1.3.10, p.~33]{BBD}. In this notation we have
$$
\text{Thick}_\mathcal{T}(\mathcal{X})=\langle \mathcal{X}\rangle_\infty:=\bigcup_{i\geq 1} \langle \mathcal{X}\rangle_i.
$$

\begin{rem}
One has to be careful that $Z^*(\mathcal{T})$ may not be a set unless $\mathcal{T}$ is essentially small (or 'skeletally small'), i.e. when $\mathcal{T}$ is locally small and the isomorphism classes of objects form a set. In what follows this will not be an issue as we will primarily be interested in the graded center of 
$\text{Thick}_\mathcal{T}(\mathcal{X})$ which {\emph{is}} essentially small by the above construction. For example, if $\mathcal{T}$ is compactly generated then $\mathcal{T}^c$ is skeletally small and hence $Z^*(\mathcal{T}^c)$ is a set.
\end{rem}

The next result is well-known but we include the proof for convenience. 

\begin{lem}\label{nilp}
Suppose $\sigma, \tau \in Z^*(\mathcal{T})$ are homogeneous and they both vanish on $\langle \mathcal{X}\rangle_i$ for some $i \geq 1$. Then their product
$\sigma \tau$ vanishes on $\langle \mathcal{X}\rangle_{i+1}$. In particular, if $t \in Z^*(\mathcal{T})$ is homogeneous and vanishes on $\mathcal{X}$ then 
$t^{2^n}$ vanishes on $\langle \mathcal{X}\rangle_{n+1}$ for all $n\geq 0$.
\end{lem}

\begin{proof}
Let $\sigma \in Z^r(\mathcal{T})$ and $\tau \in Z^s(\mathcal{T})$. Then consider a triangle $V' \rightarrow V \rightarrow V'' \rightarrow$ with $V' \in \langle \mathcal{X}\rangle_i$ 
and $V\in \langle \mathcal{X}\rangle_1$, so that $V'' \in \langle \mathcal{X}\rangle_{1}\star \langle \mathcal{X}\rangle_i$. Now look at the diagram
$$
\begin{tikzcd}
V' \arrow[r] \arrow[d, "\tau_{V'}=0"'] & V \arrow[r] \arrow[d, "\tau_V=0"'] & V'' \arrow[r] \arrow[d, "\tau_{V''}"] \arrow[dr, "0"] \arrow[dl, bend right=10, dashrightarrow]& \Sigma V' \arrow[d, "\tau_{\Sigma V'}=(-1)^s\Sigma(\tau_{V'})=0"]\\
\Sigma^s V' \arrow[r] \arrow[d, "\Sigma^s(\sigma_{V'})=0"'] & \Sigma^s V \arrow[r] \arrow[d, "\Sigma^s(\sigma_V)=0"'] \arrow[dr, "0"]& \Sigma^s V'' \arrow[r] \arrow[d, "\Sigma^s(\sigma_{V''})"] & \Sigma^{s+1} V'  \arrow[dl, bend left=10, dashrightarrow] \arrow[d, "\Sigma^s(\sigma_{\Sigma V'})=(-1)^r\Sigma^{s+1}(\sigma_{V'})=0"] \\
\Sigma^{r+s} V' \arrow[r] & \Sigma^{r+s} V \arrow[r] & \Sigma^{r+s} V'' \arrow[r] & \Sigma^{r+s+1} V' 
\end{tikzcd}
$$
There exists dashed morphisms as in the diagram since $\Hom_\mathcal{T}(V'',-)$ is a homological functor and $\Hom_\mathcal{T}(-, \Sigma^{r+s}V'')$ is a cohomological functor. Two consecutive morphisms in a triangle compose to zero, so a diagram chase shows that $(\sigma \tau)_{V''}:=\Sigma^s(\sigma_{V''})\circ \tau_{V''}=0$. Thus $\sigma\tau$ vanishes on $\langle \mathcal{X}\rangle_{1}\star \langle \mathcal{X}\rangle_i$ and therefore also on direct summands of objects in here. We conclude that $\sigma \tau$ is zero on 
$\langle \mathcal{X}\rangle_{i+1}$ as claimed.

The last statement now follows by induction on $n$. The base case $n=0$ is easy. If $t$ vanishes on $\mathcal{X}$, it vanishes on all of $\langle \mathcal{X} \rangle_1$ since $t$ commutes with $\Sigma$ up to a sign, and with coproducts. Assume  $t^{2^n}$ vanishes on $\langle \mathcal{X}\rangle_{n+1}$ for some $n\geq 0$. Taking $\sigma=\tau=t^{2^n}$ above then allows us to conclude that $t^{2^{n+1}}=t^{2^n}t^{2^n}$ indeed is zero on all objects from $\langle \mathcal{X}\rangle_{n+2}$.
\end{proof}

In our applications $\mathcal{T}$ is generated by a singleton $\mathcal{X}=\{\Bbb{X}\}$ as a thick subcategory of some ambient triangulated category. In this situation the previous result shows that all homogeneous elements $t \in Z^*(\mathcal{T})$ with $t_{\Bbb{X}}=0$ (or even just nilpotent) are locally nilpotent and therefore zero in $Z^*(\mathcal{T})_\text{lred}$. (The latter now denotes the quotient of $Z^*(\mathcal{T})$ by the homogeneous ideal generated by the locally nilpotent homogeneous elements.) 

\subsection{Comparing graded centers to Bernstein centers}

We consider the unbounded derived category $D(G):=D(\Mod(G))$ and its graded center $Z^*(D(G))$. For the time being we are mostly interested in the degree zero piece
$Z^0(D(G))$. 

Viewing $V \in \Mod(G)$ as a complex concentrated in degree zero gives an equivalence of categories between $\Mod(G)$ and the full subcategory of $D(G)$ consisting of complexes $V^\bb$ with $h^i(V^\bb)=0$ in all degrees $i \neq 0$. This point of view gives rise to a surjective restriction homomorphism of algebras
$$
\pi: Z^0(D(G)) \longrightarrow Z(\Mod(G))
$$
that admits a canonical splitting $s$ constructed in the following proof.

\begin{lem}\label{split}
The above homomorphism $\pi$ is a split surjection.
\end{lem}

\begin{proof}
Every tuple $(t_V)_{V \in \Mod(G)}$ in the center of $\Mod(G)$ yields a morphism of complexes $t_{V^\bb}$, for every complex $V^\bb$, by taking $t_{V^i}$ in degree $i$ (note how the compatibility with the differentials follows from the naturality of the tuple). It is easy to check that $t_{\Sigma V^\bb}=\Sigma(t_{V^\bb})$ so this at least produces an element of $Z^0(K(G))$. The localization functor $K(G)\rightarrow D(G)$ induces a homomorphism $Z^*(K(G))\rightarrow Z^*(D(G))$, see part (ii) of \cite[Prop.~2.3]{KY11} for example, which altogether yields an element of  
$Z^0(D(G))$. This construction gives a homomorphism of algebras $s: Z(\Mod(G))\rightarrow Z^0(D(G))$ with the property that $\pi \circ s=\text{Id}$. 
\end{proof}

Our goal is to understand the graded center of $\text{Thick}_{D(G)}(\Bbb{X}_U)$ for an arbitrary pro-$p$ open subgroup $U \subset G$. We do this by comparing it to the Bernstein center of the 'abelian' analogue $\text{Thick}_{\Mod(G)}(\Bbb{X}_U)$, i.e. the smallest thick subcategory of $\Mod(G)$ containing $\Bbb{X}_U$. We warn the reader that by a thick subcategory of $\Mod(G)$ we do \emph{not} mean a Serre subcategory, instead we adopt the terminology used in \cite{KS13} for instance:

\begin{defn}\label{thck}
$\mathcal{A}\subset \Mod(G)$ is a {\bf{thick}} subcategory if it is closed under taking direct summands and satisfies the 2-out-of-3 property on short exact sequences; that is, if 
$0\rightarrow V' \rightarrow V \rightarrow V'' \rightarrow 0$ is a short exact sequence in $\Mod(G)$ such that two of 
the terms $V',V,V''$ belong to $\mathcal{A}$ then so does the third. In other words, $\mathcal{A}$ is closed under kernels of epimorphisms, cokernels of monomorphisms, and extensions. 
\end{defn}

We summarize the basic properties of $\text{Thick}_{\Mod(G)}(\Bbb{X}_U)$ needed in the sequel. 

\begin{lem}\label{thickprops}
Let $U \subset G$ be a pro-$p$ open subgroup. Then the following holds.
\begin{itemize}
\item[(a)] $\text{Thick}_{\Mod(G)}(\Bbb{X}_U)$ contains the family of generators $\{\Bbb{X}_{U'}\}_{U'\subset U}$ for $\Mod(G)$.
\item[(b)] $\text{Thick}_{\Mod(G)}(\Bbb{X}_U)\subset \text{Serre}_{\Mod(G)}(\Bbb{X}_U)$.
\item[(c)] Restriction induces an isomorphism of $k$-algebras
$$
Z(\text{Serre}_{\Mod(G)}(\Bbb{X}_U))_\text{lred} \overset{\sim}{\longrightarrow}Z(\text{Thick}_{\Mod(G)}(\Bbb{X}_U))_\text{lred}.
$$
\item[(d)] $\text{Thick}_{\Mod(G)}(\Bbb{X}_U) \subset \Mod(G) \cap \text{Thick}_{D(G)}(\Bbb{X}_U)$.
\end{itemize}
\end{lem}

\begin{proof}
The proof of part (a) is a small alteration of the proof of Lemma \ref{fg}. Indeed $\text{Thick}_{\Mod(G)}(\Bbb{X}_U)$ contains 
$\text{ind}_U^G(F)$ for all finite-dimensional smooth $k[U]$-modules $F$, by induction on $\dim F$. This works because thick subcategories are closed under extensions. 
For $F=\text{ind}_{U'}^U(1)$ this gives what we want. 

Part (b) is trivial since Serre subcategories are thick. In particular $\text{Serre}_{\Mod(G)}(\Bbb{X}_U)$ is exactly the Serre subcategory generated by 
$\text{Thick}_{\Mod(G)}(\Bbb{X}_U)$. Therefore Lemma \ref{serre}, which applies to the additive subcategory $\mathcal{A}=\text{Thick}_{\Mod(G)}(\Bbb{X}_U)$ by part (a), implies part (c).

To show (d) it is enough to check that the intersection $\Mod(G) \cap \text{Thick}_{D(G)}(\Bbb{X}_U)$ is a thick subcategory of $\Mod(G)$. It is clearly closed under taking direct summands since $\text{Thick}_{D(G)}(\Bbb{X}_U)$ has this property, and $\Mod(G)$ is an additive subcategory of $D(G)$. To establish the 
2-out-of-3 property, note that a short exact sequence $0\rightarrow V' \rightarrow V \rightarrow V'' \rightarrow 0$ in $\Mod(G)$ gives a triangle 
$V' \rightarrow V \rightarrow C \rightarrow$ with $C \simeq V''$ in $D(G)$. As $\text{Thick}_{D(G)}(\Bbb{X}_U)$ is a strictly full triangulated subcategory, if it contains 2-out-of 
$V',V,V''$ it must contain all three. 
\end{proof}

We gather together some of the morphisms we have introduced so far in the following diagram, which commutes for trivial reasons (all the morphisms arise from restriction to various subcategories)
$$
\begin{tikzcd}
Z^0(D(G)) \arrow[rr, "\text{restriction}"] \arrow[d, twoheadrightarrow, "\pi"] & & Z^0( \text{Thick}_{D(G)}(\Bbb{X}_U)) \arrow[d, "\pi_T"] \\
Z(\Mod(G)) \arrow[r, twoheadrightarrow, "\rho"] & Z( \text{Serre}_{\Mod(G)}(\Bbb{X}_U))_\text{lred} \arrow[r, "\sim"] & Z(\text{Thick}_{\Mod(G)}(\Bbb{X}_U))_\text{lred}.
\end{tikzcd}
$$
Here $\pi_T$ comes from part (d) of Lemma \ref{thickprops}, and the lower right isomorphism comes from part (c). 

\begin{prop}\label{pit}
$\pi_T$ induces an isomorphism $Z^0(\text{Thick}_{D(G)}(\Bbb{X}_U))_\text{lred} \overset{\sim}{\longrightarrow} Z( \text{Thick}_{\Mod(G)}(\Bbb{X}_U))_\text{lred}$.
\end{prop}

\begin{proof}
From the previous diagram we at once conclude that $\pi_T$ is surjective. If $t \in Z^0( \text{Thick}_{D(G)}(\Bbb{X}_U))$ and 
$t_V$ is nilpotent for all $V \in \text{Thick}_{\Mod(G)}(\Bbb{X}_U)$ then certainly $t_{\Bbb{X}_U}$ is nilpotent. By Lemma \ref{nilp} (see the paragraph right after its proof) we find that $t$ is zero in $Z^0( \text{Thick}_{D(G)}(\Bbb{X}_U))_\text{lred}$. 
\end{proof}

In conjunction with Proposition \ref{addendum}, and part (c) of Lemma \ref{thickprops}, Proposition \ref{pit} gives the following main result:

\begin{thm}\label{mainr}
Let $G$ be a locally profinite group that contains a pro-$p$ open subgroup $U$, and let $k$ be a field of characteristic $p$. Suppose $G$ contains no proper open centralizers. 
Then there is an isomorphism of $k$-algebras
$$
k[Z(G)/Z(G)\cap U]_\text{red} \overset{\sim}{\longrightarrow} Z^0(\text{Thick}_{D(G)}(\Bbb{X}_U))_\text{lred}
$$
induced by $Z(G) \ni c \mapsto (s\circ \Phi)(c)$.
\end{thm}

In this statement $s$ denotes the section of $\pi$ constructed in the proof of Lemma \ref{split}. More precisely, the isomorphism in Theorem \ref{mainr} is the composition
$\pi_T^{-1} \circ \big(\text{restriction from (c) above}\big)\circ \Psi$.

\subsection{The case of a $p$-adic Lie group}\label{padiclie}

Our motivating example is the case of a $p$-adic Lie group $G$. In this case, under mild hypotheses, we can now compute the graded center of the subcategory of compact objects $D(G)^c$ up to locally nilpotent elements. 

\begin{cor}\label{maincor}
Let $G$ be a $p$-adic Lie group without proper open centralizers, and let $k$ be a field of characteristic $p$.
Then there is an isomorphism of $k$-algebras
$$
k[Z(G)/Z(G)\cap U]_\text{red} \overset{\sim}{\longrightarrow} Z^*(D(G)^c)_\text{lred}
$$
for any choice of pro-$p$ open subgroup $U \subset G$.
\end{cor}

\begin{proof}
We have already seen in Lemma \ref{indepz} that $k[Z(G)/Z(G)\cap U]_\text{red}$ is independent of the choice of $U$ up to isomorphism, so we may take $U$ to be torsion-free. 
This ensures that $\text{Thick}_{D(G)}(\Bbb{X}_U)=D(G)^c$ by \cite[Rem.~10, p.~456]{Sch15}. Now Theorem \ref{mainr} gives an isomorphism
$$
k[Z(G)/Z(G)\cap U]_\text{red} \overset{\sim}{\longrightarrow} Z^0(D(G)^c)_\text{lred}.
$$
It remains to verify that all $t \in Z^i(D(G)^c)$ are locally nilpotent when $i \neq 0$. As mentioned in the paragraph after Lemma \ref{nilp} it suffices to observe that 
$t_{\Bbb{X}_U}\in \Ext_{\Mod(G)}^i(\Bbb{X}_U, \Bbb{X}_U)$ is nilpotent. This is obvious for $i<0$. For $i>0$ we note that $t_{\Bbb{X}_U}^{d+1}=0$, with $d:=\dim_{\Q_p}G$, since 
$H^i(U,-)$ vanishes in degrees outside the range $[0,d]$ (as $k[\![U]\!]$ has global dimension $d$ by virtue of $U$ being torsion-free). 
\end{proof}

This applies to $G={\bf{G}}(\frak{F})$ for any connected algebraic group ${\bf{G}}$ defined over a non-archimedean local field $\frak{F}/\Q_p$. See \cite[Thm.~6.13, p.~16]{AS23}
(which even applies to smooth ${\bf{G}}$ over $\frak{F}$ of characteristic $p$). 

When ${\bf{G}}$ is a split connected reductive group over $\frak{F}/\Q_p$ there is another approach to Corollary \ref{maincor} based on \cite[Thm.~3.1.10, p.~107]{Bod22}. The latter describes the degree zero part of the graded center of the pro-$p$ Iwahori $\Ext$-algebra $\Ext_{\Mod(G)}^*(\Bbb{X}_I, \Bbb{X}_I)$. Each $c \in Z(G)$ trivially gives an element $\tau_c=\text{char}_{IcI}\in Z^0(\Ext_{\Mod(G)}^*(\Bbb{X}_I, \Bbb{X}_I))$ and Bodon proves that this results in an isomorphism
$$
k[Z(G)/Z(G)\cap I] \overset{\sim}{\longrightarrow} Z^0(\Ext_{\Mod(G)}^*(\Bbb{X}_I, \Bbb{X}_I)).
$$
Moreover, this algebra is reduced by \cite[Rem.~3.1.2, p.~101]{Bod22} which explains why $Z(G)/Z(G)\cap I \simeq \Z^r\times A$ for some $r\geq 0$ and some finite abelian group $A$ such that $p \nmid |A|$. Therefore the group algebra is a subalgebra of $k(X_1,\ldots,X_r)[A]$ which is semisimple. 

On the other hand, under the assumption that $I$ is torsion-free so that $\Bbb{X}_I$ is in $D(G)^c$, the evaluation map $t \mapsto t_{\Bbb{X}_I}$ gives an injective homomorphism
$$
Z^0(D(G)^c)_\text{lred} \hookrightarrow Z^0(\Ext_{\Mod(G)}^*(\Bbb{X}_I, \Bbb{X}_I))
$$
as we have seen after Lemma \ref{nilp}. By Bodon's result this is necessarily also surjective, as follows from the commutative diagram
\begin{equation}\label{inidiag}
\begin{tikzcd}
Z(\Mod(G)) \arrow[r, "s"] & Z^0(D(G)) \arrow[r, "\text{res}"] & Z^0(D(G)^c) \arrow[r, "(\cdot)_{\Bbb{X}_I}"] & Z^0(\Ext_{\Mod(G)}^*(\Bbb{X}_I, \Bbb{X}_I)) \\
\widehat{k[Z(G)]} \arrow[rrr, twoheadrightarrow, "\widehat{\text{pr}}"] \arrow[u, "\Phi"] & & & k[Z(G)/Z(G)\cap I].  \arrow[u, "c\mapsto \tau_c", "\simeq"']
\end{tikzcd}
\end{equation}
Here $s$ denotes the splitting of $\pi$ from the proof of Lemma \ref{split}.

\begin{lem}
The previous diagram (\ref{inidiag}) commutes.
\end{lem}

\begin{proof}
Since $\pi \circ s=\text{Id}$, by the proof of Lemma \ref{split}, the composition in the top row is just the evaluation map
$Z(\Mod(G)) \rightarrow \End_{\Mod(G)}(\Bbb{X}_I)$. Starting with an $x \in \widehat{k[Z(G)]}$ and going up-then-right produces the endomorphism 
$\Phi_{\Bbb{X}_I}(x)$. 

Let $x=(x_{Z'})$ where 
$x_{Z'}\in k[Z(G)/Z']$ and $Z'$ runs over the compact open subgroups of $Z(G)$. We are interested in the component
$\widehat{\text{pr}}(x)=x_{Z(G)\cap I}=\sum_c a_cc$ where $c$ runs over a set of representatives for $Z(G)/Z(G)\cap I$.
What we have to check is that
$$
\sum_ca_c\tau_c=\Phi_{\Bbb{X}_I}(x).
$$ 
Since $Z(G)\cap I$ acts trivially on $\Bbb{X}_I$, its $\widehat{k[Z(G)]}$-module structure factors through $k[Z(G)/Z(G)\cap I]$, and the formula follows. 
\end{proof}


\section{Group cohomology as a source of central elements}

Let $G$ be a locally pro-$p$ group. As in \cite[Cor.~3.3, p.~42]{SS23} we may endow $D(G)$ with a monoidal structure $\otimes$ and thereby promote it to a tensor triangulated category with unit $k$. As a result we have a homomorphism of graded $k$-algebras
\begin{align*}
\varphi: \Ext_{\Mod(G)}^*(k,k)=\bigoplus_{r \in \Z} \Hom_{D(G)}(k, \Sigma^r k) &\longrightarrow Z^*(D(G)) \\
\eta & \longmapsto \big(\eta \otimes \text{Id}_{V^\bb}\big)_{V^\bb \in D(G)}.
\end{align*}
This is a right inverse of the evaluation-at-$k$ map, so in particular $\varphi$ is a split injection. We note that $\Ext_{\Mod(G)}^r(k,k)$ is isomorphic to the continuous group cohomology 
$H_\text{cts}^r(G,k)$ by \cite[Thm.~1.1]{Fus22}, but we will not make use of this fact for non-compact groups. The map $\varphi$ often gives a way to produce examples of non-trivial elements of $Z^*(D(G))$ in positive degrees. For finite $G$ the 'canonical map' $\varphi$ was also employed in \cite[Prop.~1.3]{Lin09}.

\begin{exmp}
Consider $G=\SL_2(\Q_p)$ with $p \geq 5$. In this case $\Ext_{\Mod(G)}^*(k,k)$ is fully computable. Since $G$ is the amalgamation of $K=\SL_2(\Z_p)$ and 
$K'=\begin{psmallmatrix}p & \\ & 1\end{psmallmatrix}^{-1}K\begin{psmallmatrix}p & \\ & 1\end{psmallmatrix}$ over $J=K\cap K'$ we have a long exact sequence
$$
\cdots \Ext_{\Mod(G)}^r(k,k) \longrightarrow \Ext_{\Mod(K)}^r(k,k) \times \Ext_{\Mod(K')}^r(k,k) \longrightarrow \Ext_{\Mod(J)}^r(k,k)  \longrightarrow \Ext_{\Mod(G)}^{r+1}(k,k) \cdots
$$
from \cite[Thm.~1.1]{Sch26}.
The reference \cite[Thm.~1.1]{DL25} gives the following:
\begin{itemize}
\item The restriction map $H^*(K,k) \longrightarrow H^*(J,k)$ is an isomorphism;
\item $H^*(J,k)$ is a free graded-commutative $k$-algebra on a generator of degree $3$, i.e.
$$
H^*(K,k) \underset{\text{res}}{\overset{\sim}{\longrightarrow}} H^*(J,k) \simeq k[v]/(v^2)
$$
where $\text{deg}(v)=3$.
\end{itemize}
The above long exact sequence therefore only has nonzero terms for $r=0,3$. For $r=3$ one infers that 
$\Ext_{\Mod(G)}^3(k,k)$ must be one-dimensional. We conclude that $\Ext_{\Mod(G)}^*(k,k)$ is concentrated in degrees $\{0,3\}$; in fact it is necessarily isomorphic to $k[v]/(v^2)$ as a graded $k$-algebra (where $v$ still has degree $3$).  

Via $\varphi$ this gives an example of a nonzero $t \in Z^3(D(G))$ with $t^2=0$. 
\end{exmp}


\section{An application to block decompositions}\label{appblock}

In this section, as an application of Corollary \ref{maincor}, we factor $D(G)$ into a direct product of finitely many \emph{indecomposable} subcategories under the additional assumption that $Z(G)$ is topologically finitely generated. See Proposition \ref{indecder} below. In Lemma \ref{topfingen} we verify that this applies to all $p$-adic reductive groups.

\subsection{Various finiteness properties}

We start with the following general remark.

\begin{lem}
Let $Z$ be a locally profinite group, assumed to be abelian and topologically finitely generated. Then $Z$ has a unique maximal compact subgroup $Z_0$. This $Z_0$ is 
open and profinite, $Z/Z_0$ is a free abelian group of some finite rank $r$, and there is a non-canonical isomorphism of topological groups  
$Z_0 \times \Z^r \overset{\sim}{\longrightarrow} Z$.
\end{lem}

\begin{proof}
Choose an open subgroup $U \subset Z$ that is profinite in the induced topology. Then $Z/U$ is a finitely generated abelian group and we write its torsion subgroup as $U'/U$ for some $U'\supset U$ which is necessarily also open and profinite. Then $Z/U'$ is free of some finite rank $r$ and we identify it with $\Z^r$. The choice of a splitting $s: \Z^r \rightarrow Z$ gives rise to an isomorphism of topological groups $U' \times \Z^r \overset{\sim}{\longrightarrow} Z$ sending $(z,\lambda)\mapsto z\cdot s(\lambda)$. (This is clearly a continuous bijection and it takes open sets to open sets.)

It is now immediate that $U'$ is maximal among all the compact subgroups of $Z$, for if $U'' \supset U'$ is compact the image of $U''/U'$ in $\Z^r$ must be trivial. It is the only one, for if $\widetilde{U}$ is any maximal compact subgroup its image in $\Z^r$ is trivial, i.e. $\widetilde{U} \subseteq U'$, and this must be an equality by maximality. 
\end{proof}

We keep the same $Z$, abelian and topologically finitely generated, but we now assume in addition that $Z$ is locally pro-$p$ and consider $\widehat{k[Z]}$ for some field $k$ of characteristic $p$. 

Let $Z_0\subset Z$ be the maximal compact subgroup, $P \subset Z_0$ be the unique Sylow pro-$p$ subgroup, and let $Q \subset Z_0$ be the unique finite subgroup of order prime-to-$p$ such that $Z_0=P\times Q$. Correspondingly 
we have a (non-canonical) isomorphism of $k$-algebras
\begin{align*}
\widehat{k[Z]} \simeq \widehat{k[Z_0 \times \Z^r]}&={\varprojlim}_{V \subset P} k[P/V \times Q \times \Z^r] \\
&={\varprojlim}_{V \subset P} \big(k[P/V] \otimes_k k[Q] \otimes_k k[\Z^r] \big)\\
&=\big({\varprojlim}_{V \subset P} k[P/V] \otimes_k k[\Z^r] \big) \otimes_k k[Q] \\
&=\widehat{k[P \times \Z^r]} \otimes_k k[Q].
\end{align*}
(Here $V$ runs over the compact open subgroups of $P$, and we can take $k[Q]$ out of the limit because it is finite-dimensional.) 

To better understand the algebra $\widehat{k[P \times \Z^r]}$
we now pass to the fraction field $k(\Z^r)$ of $k[\Z^r]$. We note that $k[\Z^r]$ is a Laurent polynomial algebra 
$k[X_1^{\pm 1},\ldots, X_r^{\pm 1}]$ so the fraction field $k(\Z^r)$ is just shorthand notation for the field of rational functions $k(X_1,\ldots, X_r)$. Continuing the thread above we have an embedding of $k$-algebras
$$
\widehat{k[P \times \Z^r]}={\varprojlim}_{V \subset P} \big(k[P/V] \otimes_k k[\Z^r] \big) \hookrightarrow 
{\varprojlim}_{V \subset P} \big(k[P/V] \otimes_k k(\Z^r) \big)=k(\Z^r)[\![P]\!].
$$
The target $k(\Z^r)[\![P]\!]$ is a local ring since $P$ is pro-$p$ and the coefficient field $k(\Z^r)$ is of characteristic $p$ (see the proof of part i of \cite[Prop.~19.7, p.~162]{Sch11} for example); in particular it has no idempotents $e$ other than $0$ and $1$ -- we cannot have both $e$ and $1-e$ in the maximal ideal, so either $e$ is invertible or $1-e$ is.
Therefore the subalgebra $\widehat{k[P \times \Z^r]}$ certainly does not have any idempotents other than $0$ and $1$.

Note that $k[Q]$ is semisimple, so it is a product
$\prod_{i \in I} k_i$ of finite field extensions $k_i/k$. Thus we can decompose $\widehat{k[Z]}$ further and get a (non-canonical) isomorphism 
$$
\widehat{k[Z]} \simeq \prod_{i \in I} \widehat{k_i[P \times \Z^r]}.
$$
This allows us to give a precise count of the number of idempotents in $\widehat{k[Z]}$; indeed from our previous observations the right-hand side has exactly 
$2^{|I|}$ idempotents. Apropos of this, here is a related result. 

\begin{lem}\label{idempbij}
$\widehat{k[Z]}_\text{red}$ is a Noetherian ring, and the projection $\widehat{k[Z]} \twoheadrightarrow \widehat{k[Z]}_\text{red}$ restricts to a bijection 
$$
\{\text{idempotents in $\widehat{k[Z]}$}\} \overset{1:1}{\longleftrightarrow} \{\text{idempotents in $\widehat{k[Z]}_\text{red}$}\}.
$$
These sets are also in bijection with the set of all idempotents in $k[Q]$.
\end{lem}

\begin{proof}
Observe that 
$\widehat{k[Z]}_\text{red}\simeq \prod_{i \in I} k_i[\Z^r]$ by taking the $Z'$ in Lemma \ref{indepz} to be $P$ viewed as a subgroup of $P \times \Z^r$. The set of idempotents in $\widehat{k[Z]}_\text{red}$ is therefore in bijection with the set of tuples $(\varepsilon_i)_{i \in I}$ with components $\varepsilon_i\in \{0,1\}$, similarly to what happens for $\widehat{k[Z]}$ and $k[Q]$.
\end{proof}

When $Z$ is assumed to be a $p$-adic Lie group, the whole completion $\widehat{k[Z]}$ itself is even a Noetherian ring (as opposed to just its topological nilreduction) and this applies to the center of a $p$-adic reductive group as the next result shows.

\begin{lem}\label{topfingen}
Let $k$ be a field of characteristic $p$. Then the following is true. 
\begin{itemize}
\item[(a)] If $Z$ is a $p$-adic Lie group, assumed to be abelian and topologically finitely generated, then $\widehat{k[Z]}$ is Noetherian;
\item[(b)] Let ${\bf{T}}$ be an algebraic torus defined over a finite extension $\frak{F}/\Q_p$. Then $T:={\bf{T}}(\frak{F})$ is topologically finitely generated; more generally --
\item[(c)] Let ${\bf{G}}$ be a reductive group defined over a finite extension $\frak{F}/\Q_p$, and let $G:={\bf{G}}(\frak{F})$. Then the center $Z(G)$ is 
topologically finitely generated and $\widehat{k[Z(G)]}$ is Noetherian.
\end{itemize}
\end{lem}

\begin{proof}
An abelian uniform pro-$p$ group of rank $d$ is isomorphic to $\Z_p^d$ as a topological group. (The homeomorphism $c$ on \cite[p.~195]{Sch11} respects the group structure in the abelian case.)

Let $A$ be a commutative $k$-algebra. Then $A[\![\Z_p^d]\!]$ is isomorphic to the power series algebra 
$A[\![X_1,\ldots, X_d]\!]$. Indeed, if we let $g_i:=(0,\ldots,1,\ldots,0) \in \Z_p^d$ with $1$ in the $i^\text{th}$ slot, then the map $X_i \mapsto g_i-1$ gives compatible isomorphisms 
$$
k[X_1,\ldots,X_d]/(X_1^{p^n},\ldots, X_d^{p^n}) \overset{\sim}{\longrightarrow} k[(\Z/p^n\Z)^d]
$$
for all $n$. (They are clearly surjective, and the two algebras above both have dimension $p^{nd}$ over $k$.) Now tensor by $A$ over $k$ and pass to the limit. This results in an isomorphism of $A$-algebras
$$
A[\![X_1,\ldots, X_d]\!]=\varprojlim_n A[X_1,\ldots,X_d]/(X_1^{p^n},\ldots, X_d^{p^n})  \overset{\sim}{\longrightarrow} 
\varprojlim_n A[(\Z/p^n\Z)^d]=A[\![\Z_p^d]\!].
$$
In particular $A[\![\Z_p^d]\!]$ is Noetherian when $A$ is Noetherian. Applying this to the coefficient ring $A:=k[\Z^r]$ yields part (a) as follows. Since $Z_0$ is a compact $p$-adic Lie group it contains a uniform pro-$p$ open subgroup $Z_1\simeq \Z_p^d$ of finite index. As $A[\![Z_0]\!]$ is finitely generated as an $A[\![Z_1]\!]$-module, and $A[\![Z_1]\!]\simeq A[\![X_1,\ldots, X_d]\!]$
is Noetherian, 
$$
\widehat{k[Z]}\simeq \varprojlim_{V\subset Z_0} \big(k[Z_0/V]\otimes_k k[\Z^r]\big)=A[\![Z_0]\!]
$$ 
is Noetherian as well.

For part (b) let ${\bf{T}}_s$ be the maximal $\frak{F}$-split subtorus of ${\bf{T}}$, and let ${\bf{T}}_a$ be the maximal $\frak{F}$-anisotropic subtorus of ${\bf{T}}$. The multiplication map gives an isogeny ${\bf{T}}_a \times {\bf{T}}_s\rightarrow {\bf{T}}$ with kernel ${\bf{K}}$ isomorphic to a product $\mu_{n_1}\times \cdots \times \mu_{n_l}$ for some $n_i$. The long exact sequence in Galois cohomology starts off like 
$$
0 \longrightarrow K \longrightarrow T_a \times T_s \longrightarrow T \longrightarrow H^1(\frak{F}, {\bf{K}}(\overline{\frak{F}})) \longrightarrow \cdots,
$$
where $T_s:={\bf{T}}_s(\frak{F})$ and so on. Clearly $K$ is a finite group, and so is the term $H^1(\frak{F}, {\bf{K}}(\overline{\frak{F}}))$. Indeed ${\bf{K}}(\overline{\frak{F}})$ is isomorphic to 
$\mu_{n_1}(\overline{\frak{F}})\times \cdots \times \mu_{n_l}(\overline{\frak{F}})$ as Galois-modules, and $H^1(\frak{F},-)$ preserves direct sums, so we are reduced to noting that 
$H^1(\frak{F},\mu_{n}(\overline{\frak{F}}))\simeq \frak{F}^\times/ \frak{F}^{\times n}$ is finite for all $n$ since $\frak{F}$ is a $p$-adic field. 

Now, knowing that $T_a \times T_s \rightarrow T$ has finite cokernel, we just have to verify that $T_s$ and $T_a$ are both topologically finitely generated. The split case follows immediately from the one-dimensional case where $\frak{F}^\times\simeq \Z \times \mu_\infty(\frak{F}) \times \Z_p^{[\frak{F}:\Q_p]}$ is visibly topologically finitely generated (again because $\frak{F}$ is a characteristic zero local field). For the anisotropic case we note that $T_a$ is a compact $p$-adic Lie group. It therefore contains a uniform pro-$p$ open subgroup, necessarily of finite index, and in particular $T_a$ is topologically finitely generated (since uniform groups have this property). 

Here is an alternative proof of (b) to cater to the taste of the reader. Suppose ${\bf{T}}$ splits over the finite extension $\frak{E}/\frak{F}$ and identify 
${\bf{T}}\times \frak{E}$ with $\Bbb{G}_m^\ell$. Note that $T$ is a closed subgroup of ${\bf{T}}(\frak{E})$ so it suffices to show that every closed subgroup $S$ of 
$(\frak{E}^\times)^\ell=\frak{E}^\times \times \cdots \times \frak{E}^\times$ is topologically finitely generated. We choose an $N$ large enough that exp and log restrict to mutually inverse isomorphisms $\frak{m}_\frak{E}^N \overset{\sim}{\longrightarrow} 1+\frak{m}_\frak{E}^N$. Now $\frak{E}^\times/1+\frak{m}_\frak{E}^N$ is clearly a finitely generated abelian group, and therefore the subgroup
$$
S/ \big(S \cap (1+\frak{m}_\frak{E}^N)^\ell \big) \hookrightarrow (\frak{E}^\times/1+\frak{m}_\frak{E}^N)^\ell
$$
is finitely generated as well. Therefore it is enough to show that $S \cap (1+\frak{m}_\frak{E}^N)^\ell$ is topologically finitely generated, but via exp and log this subgroup corresponds to a finitely generated $\Z_p$-submodule of $(\frak{m}_\frak{E}^N)^\ell$.

Moving on to part (c), for the purpose of this proof we let ${\bf{C}}$ denote the full center of ${\bf{G}}$ and consider the identity component ${\bf{C}}^\circ$. Since 
${\bf{C}}^\circ$ is a torus we know from part (b) that ${\bf{C}}^\circ(\frak{F})$ is topologically finitely generated, hence so is ${\bf{C}}(\frak{F})=Z(G)$ because the index is finite. Part (a) now allows us to deduce that indeed $\widehat{k[Z(G)]}$ is a Noetherian ring.
\end{proof}

We note that $\Q_p$ is an example of an abelian $p$-adic Lie group which is not topologically finitely generated (a finite subset is contained in $p^n\Z_p$ for some $n$). 

\subsection{Blocks of smooth representations}

We return to the case of an arbitrary locally pro-$p$ group $G$ and a coefficient field $k$ of characteristic $p$. In this section we assume $Z:=Z(G)$ is topologically finitely generated. 
As above we identify $Z$ with $Z_0\times \Z^r$ and factor the maximal compact subgroup of the center as $Z_0=P \times Q$. 
We note that every continuous character $\chi:Z_0\rightarrow k^\times$ is automatically trivial on $P$, and we interchangeably think of $\chi$ as a character 
$\chi: Q\rightarrow k^\times$.

Recall that $k[Q]\simeq \prod_{i\in I}k_i$ and we let $(e_i)_{i \in I}$ denote the set of primitive idempotents in $k[Q]$, corresponding to $(0,\ldots,1,\ldots,0)$ with $1$ in the $i^\text{th}$ position. They are mutually orthogonal and $\sum_{i \in I}e_i=1$.

\begin{rem}
When $k$ is large enough, in the sense that every character $\chi: Q\rightarrow \overline{k}^\times$ takes values in $k$, these idempotents are given by the standard formula 
$e_\chi=\frac{1}{|Q|}\cdot \sum_{g \in Q} \chi(g)^{-1}g$. However, our field $k$ will remain arbitrary (of characteristic $p$). 
\end{rem} 

For each $i \in I$ we consider the largest subspace
$e_i V\subseteq V$ on which $e_i$ acts as the identity. Clearly $e_{i'}$ annihilates $e_i V$ for $i \neq i'$. Any $V \in \Mod(G)$ decomposes canonically as a direct sum of $G$-subrepresentations
$V=\bigoplus_{i\in I} e_i V$, and clearly $\Hom_{\Mod(G)}(e_i V, e_{i'}V')=0$ for $i \neq i'$, so this in fact gives a factorization of the category $\Mod(G)$ into a direct product 
\begin{equation}\label{factormod}
\Mod(G)=\prod_{i \in I} \Mod^{e_i}(G)
\end{equation}
of the full subcategories $\Mod^{e_i}(G)$ (the objects of which are those $V$ such that $e_i$ acts as the identity, i.e. $e_i V=V$). Similarly $\widehat{k[Z]}$ decomposes as a finite direct sum of ideals $\widehat{k[Z]}=\bigoplus_{i \in I} e_i \widehat{k[Z]}$. Here the algebra $e_i \widehat{k[Z]}$ is not a subalgebra of $\widehat{k[Z]}$ unless they share the same identity element, i.e. $e_i=1$.   


\begin{lem}\label{idemp}
The following hold:
\begin{itemize}
\item[(a)] For all $i \in I$ there is an isomorphism of $k$-algebras 
$$
\widehat{k_i[P \times \Z^r]}\overset{\sim}{\longrightarrow}e_i \widehat{k[Z]};
$$
\item[(b)] $\widehat{k_i[P \times \Z^r]}$ has topological nilreduction $\widehat{k_i[P \times \Z^r]}_\text{red} \overset{\sim}{\longrightarrow} k_i[\Z^r]$;
\end{itemize}
\end{lem}

\begin{proof}
Part (a) is clear from the isomorphism $\widehat{k[Z]} \simeq \widehat{k[P \times \Z^r]} \otimes_k k[Q]$ as discussed previously. For part (b) we note that 
$P$ is a pro-$p$ open subgroup of $P \times \Z^r$ so by our general observations, notably Lemma \ref{indepz}, the algebra $\widehat{k_i[P \times \Z^r]}$ has topological nilreduction 
$k_i[(P \times \Z^r)/P]_\text{red}=k_i[\Z^r]$. 
\end{proof}

We now relate the center of $\Mod^{e_i}(G)$ to the block algebra $e_i \widehat{k[Z]}$. Although the latter is usually not a subalgebra of $\widehat{k[Z]}$ we observe that the composition 
$$
\Phi^{e_i}: e_i\widehat{k[Z]} \hookrightarrow \widehat{k[Z]} \overset{\Phi}{\longrightarrow} Z(\Mod(G)) \overset{\text{res}^{e_i}}{\longrightarrow} Z(\Mod^{e_i}(G))
$$
is nevertheless a $k$-algebra homomorphism. Obviously $\Phi^{e_i}$ preserves addition and multiplication, but the point is that it takes the identity element $e_i$ to the identity element of $Z(\Mod^{e_i}(G))$ exactly because $e_i$ acts as the identity on every object of the subcategory $\Mod^{e_i}(G)$ by its very definition. 

The following gives a supplement to \cite[Thm.~1.1]{AS23} when $Z$ is topologically finitely generated. 

\begin{prop}\label{indecmod}
Let $G$ be a locally pro-$p$ group with a topologically finitely generated center $Z$, and no proper open centralizers. Let $k$ be a field of characteristic $p$.
Then: 
\begin{itemize}
\item[(a)] For all $i \in I$ the map $\Phi^{e_i}$ is an isomorphism of $k$-algebras 
$$
e_i\widehat{k[Z]} \overset{\sim}{\longrightarrow} Z(\Mod^{e_i}(G));
$$
\item[(b)] $\Mod^{e_i}(G)$ is indecomposable, i.e. it is not equivalent to a product of two non-trivial abelian categories.
\end{itemize}
\end{prop}

\begin{proof}
Let $s \in \widehat{k[Z]}$ and expand it as a finite sum $s=\sum_{j \in I} s_j$ with $s_j=e_j s \in e_j\widehat{k[Z]}$.
Similarly let $V=\bigoplus_{j \in I} V_j$ be an object of $\Mod(G)$ decomposed into $V_j=e_j V \in \Mod^{e_j}(G)$. Correspondingly 
$\Phi_V(s)=\bigoplus_{j \in I} \Phi_{V_j}(s)=\bigoplus_{j \in I} \Phi_{V_j}(s_j)$ since $e_{j'}$ annihilates $V_j$ for $j'\neq j$. In particular 
$\Phi_V(s_i)$ is the extension of $\Phi_{V_i}(s_i)$ by zero endomorphisms on the other direct summands. Thus $\Phi^{e_i}(s_i)=0$ implies $\Phi_V(s_i)=0$ for all $V$, and therefore $s_i=0$ by \cite[Thm.~1.1]{AS23}, so $\Phi^{e_i}$ is injective. 

For the surjectivity of $\Phi^{e_i}$ first note that $\text{res}^{e_i}$ is surjective; indeed it is the projection 
$$
Z(\Mod(G)) \overset{\sim}{\longrightarrow} \prod_{j \in I} Z(\Mod^{e_j}(G)) \twoheadrightarrow Z(\Mod^{e_i}(G)).
$$
Starting with $t=(t_W)_{W\in \Mod^{e_i}(G)}$ in $Z(\Mod^{e_i}(G))$ we may then express it as $t=(\text{res}^{e_i} \circ \Phi)(s)$ for some 
$s \in \widehat{k[Z]}$ again appealing to \cite[Thm.~1.1]{AS23}. This means $t_W=\Phi_W(s)=\Phi_W(s_i)$ in the notation of the previous paragraph, and we can therefore replace $s$ by $s_i$ and get the desired preimage of $t$ in $e_i\widehat{k[Z]}$. This settles part (a). 

Part (b) is now an immediate application of Lemma \ref{idemp}, for suppose there is an equivalence of abelian categories
$\Mod^{e_i}(G) \simeq \mathcal{A} \times \mathcal{B}$. Then its center is isomorphic to the product of rings $Z(\mathcal{A}) \times Z(\mathcal{B})$. If 
$\mathcal{A}$ and $\mathcal{B}$ are both non-trivial their centers contain the zero-tuple and the identity-tuple, and therefore $Z(\mathcal{A})$ and $Z(\mathcal{B})$ are nonzero. Thus 
$(1,0)$ and $(0,1)$ correspond to non-trivial idempotents of $Z(\Mod^{e_i}(G))$, but the latter is isomorphic to $\widehat{k_i[P \times \Z^r]}$ which does not have any idempotents other than $0$ and $1$.
\end{proof}

Informally part (b) of Proposition \ref{indecmod} says that the factorization (\ref{factormod}) is optimal. Our next goal is establish the analogue in the derived setting. 


\subsection{Blocks in the derived setting}

Since $\Mod^{e_i}(G)$ is a Serre subcategory the following defines a strictly full triangulated subcategory of the derived category:
$$
D^{e_i}(G):=D_{\Mod^{e_i}(G)}(G)=\{V^\bb: \text{$e_i$ acts as the identity $h^*(V^\bb)$}\}.
$$
In fact it is thick and even localizing. Clearly any complex $V^\bb$ decomposes as a direct sum of subcomplexes $V^\bb=\bigoplus_{i \in I} e_i V^\bb$, and moreover we have $\Hom_{D(G)}(e_i V^\bb, e_{i'}V'^\bb)=0$ for $i \neq i'$ (to see this take $V'^\bb$ to be homotopically injective). In other words we have $D^{e_i}(G) \simeq D(\Mod^{e_i}(G))$ as well as an analogue of (\ref{factormod})
for the derived category, i.e.
\begin{equation}\label{factorder}
D(G)=\prod_{i \in I} D^{e_i}(G).
\end{equation}
The right-hand side is the product category with suspension and distinguished triangles defined componentwise. This leads to the following analogue of Proposition \ref{indecmod} in the derived setting, but only for $p$-adic Lie groups. 

\begin{prop}\label{indecder}
Let $G$ be a $p$-adic Lie group with a topologically finitely generated center $Z$, and no proper open centralizers. Let $k$ be a field of characteristic $p$, and
fix an $i \in I$. 
Then: 
\begin{itemize}
\item[(a)] $Z^0(D^{e_i}(G)^c)_\text{lred} \overset{\sim}{\longrightarrow} k_i[\Z^r]$;
\item[(b)] $D^{e_i}(G)$ is generated by $D^{e_i}(G)^c$ as a localizing subcategory;
\item[(c)] $D^{e_i}(G)$ is indecomposable, i.e. it is not equivalent to a product of two non-trivial triangulated categories.
\end{itemize}
\end{prop}

\begin{proof}
The factorization (\ref{factorder}) continues to hold for the subcategories of compact objects since $D(G)^c$ is thick. Taking $Z^0(-)$ respects the direct product (see \cite[Ex.~2.4, p.~446]
{KY11} for instance) and so does passing to lred. Therefore we have an isomorphism of $k$-algebras
$$
Z^0(D(G)^c)_\text{lred} \overset{\sim}{\longrightarrow} \prod_{j \in I} Z^0(D^{e_j}(G)^c)_\text{lred}. 
$$
Comparing this to the factorization $\widehat{k[Z]}_\text{red}\simeq \prod_{j \in I} k_j[\Z^r]$ and using Corollary \ref{maincor} proves (a).

For part (b) we recall that $D(G)$ is generated by $\Bbb{X}_U$ for any choice of a torsion-free pro-$p$ open subgroup $U \subset G$, and $\Bbb{X}_U$ belongs to $D(G)^c$. From this it is easy to see that $e_i \Bbb{X}_U$ belongs to $D^{e_i}(G)^c$ and the localizing subcategory $e_i \Bbb{X}_U$ generates is all of $D^{e_i}(G)$.

For part (c) suppose $D^{e_i}(G)\simeq \mathcal{S}\times \mathcal{T}$ for two triangulated categories $\mathcal{S}, \mathcal{T}$. Passing to compact objects we get an equivalence
$D^{e_i}(G)^c\simeq \mathcal{S}^c\times \mathcal{T}^c$ (note that $\mathcal{S}, \mathcal{T}$ necessarily admit small coproducts, since $D^{e_i}(G)$ does, and it therefore makes sense to talk about the compact objects of $\mathcal{S}, \mathcal{T}$). We deduce that $Z^0(D^{e_i}(G)^c)_\text{lred}\simeq k_i[\Z^r]$ is isomorphic to 
$Z^0(\mathcal{S}^c)_\text{lred}\times Z^0(\mathcal{T}^c)_\text{lred}$ (see \cite[Ex.~2.4, p.~446]
{KY11} again). We conclude that $Z^0(\mathcal{T}^c)_\text{lred}=\{0\}$ by interchanging $\mathcal{S}, \mathcal{T}$ if necessary. In particular the identity-tuple 
$(\text{Id}_{V^\bb})_{V^\bb \in \mathcal{T}^c}$ is locally nilpotent, which amounts to having $\mathcal{T}^c=\{\text{zero objects}\}$. However, $\mathcal{T}$ is generated by 
$\mathcal{T}^c$ by part (b), so in fact  $\mathcal{T}=\{\text{zero objects}\}$. In other words $\mathcal{T}$ is a 'trivial' triangulated category. 
\end{proof}

\begin{exmps}
In the context of Proposition \ref{indecder}, $D(G)$ is indecomposable if $Z(G)$ is trivial -- or more generally if the pro-order of $Z(G)$ is a $p$-power (that is, when the maximal compact subgroup of $Z(G)$ is pro-$p$). 

If $G=\SL_2(\frak{F})$ for a finite extension $\frak{F}/\Q_p$ with $p>2$, we have $Z(G)=\{\pm I\}=Q$. Here $k[Q]$ has the two primitive idempotents 
$e_+$ and $e_-$, and $D(G)=D^{e_+}(G) \times D^{e_-}(G)$ where $D^{e_{\pm}}(G)$ is the indecomposable subcategory of complexes $V^\bb$ such that $-I$ acts by $\pm 1$ on 
$h^*(V^\bb)$ respectively. (In the $p=2$ case $D(G)$ is indecomposable.)

If $G=\GL_2(\frak{F})$ for a finite extension $\frak{F}/\Q_p$, with no restriction on $p$, we have $P=1+\m_\frak{F}$ and $Q=\mu_{q-1}(\frak{F})$ where $q$ is the size of the residue field, which we will denote by $\F_q$. For simplicity we assume $k$ is large enough to admit a field embedding $\sigma: \F_q \hookrightarrow k$. Then there are exactly $q-1$ characters $\chi_i:Q \rightarrow k^\times$ given by $Q\simeq \F_q^\times \overset{\sigma^i}{\longrightarrow} k^\times$ with $i=0,1,\ldots,q-2$. Correspondingly $D(G)$ factors as a direct product of $q-1$ indecomposable subcategories
$$
D(G)=D^{e_{\chi_0}}(G) \times \cdots \times D^{e_{\chi_{q-2}}}(G).
$$
Here $D^{e_{\chi_i}}(G)$ in particular contains all $V^\bb$ for which the central subgroup $\frak{O}_\frak{F}^\times$ acts on $h^*(V^\bb)$ via the inflation of the character $\chi_i$. 
\end{exmps}


\section{A comparison with the image of the characteristic homomorphism}\label{hoch}

In this section we again assume that $G$ is a locally profinite group containing a pro-$p$ open subgroup $U$, and that $G$ contains no proper open centralizer. We repeat that in this generality we so far have the maps
\begin{equation}\label{diag:Bod}
\begin{tikzcd}
k[Z(G)/Z(G)\cap U]_\text{red} \arrow[r, "\simeq"] & Z^0(\text{Thick}_{D(G)}(\Bbb{X}_U))_\text{lred} \arrow[d, hook, "t \mapsto t_{\Bbb{X}_U}"] \\
& Z^0(\Ext_{\Mod(G)}^*(\Bbb{X}_U, \Bbb{X}_U))_\text{red}
\end{tikzcd}
\end{equation}
with the horizontal isomorphism being given by Theorem \ref{mainr}, and the perpendicular evaluation map being injective as explained in the paragraph after Lemma \ref{nilp}. In the following we will determine the image of the composed map. 

We have the characteristic map out of the zeroth Hochschild cohomology of the differential graded Hecke algebra $\HH_U^\bb$, whose definition we will recall in Section \ref{heckedga} below, 
$$
\chi_U: HH^0(\HH_U^\bb) \longrightarrow Z^0(\Ext_{\Mod(G)}^*(\Bbb{X}_U, \Bbb{X}_U)).
$$
It is defined by sending a morphism $\xi: \HH_U^\bb \rightarrow \HH_U^\bb$ in the derived category of dg bimodules to the class in $h^0(\HH_U^\bb)$ represented by
the cocycle $\xi(1)$. We will give more details below. 

In this section we show that the image of $\chi_U$ coincides with the image of the composition in (\ref{diag:Bod}) up to nilpotent elements. 

\begin{thm}\label{imchar}
With the above assumptions on $G$ the composed map in (\ref{diag:Bod}) induces an isomorphism of $k$-algebras
$$
k[Z(G)/Z(G)\cap U]_\text{red} \overset{\sim}{\longrightarrow} \im(\chi_U)_\text{red}
$$
induced by $Z(G) \ni c \mapsto \tau_c$.
\end{thm}

We will also show that $\chi_U$ is surjective when $Z^0(\Ext_{\Mod(G)}^*(\Bbb{X}_U, \Bbb{X}_U))$ is spanned by all the $\tau_c=\text{char}_{IcI}$ (as is the case in the setup of \cite{Bod22} for instance).  


\subsection{Differential graded Hecke algebras}\label{heckedga}

The results in this subsection apply in the most general case. Thus $G$ now denotes an arbitrary locally profinite group, $k$ is any field, and $D(G)=D(\Mod(G))$ is the derived category of smooth $G$-representations on $k$-vector spaces. Once and for all we fix a compact open subgroup $U \subset G$ and consider the compactly induced representation 
$\Bbb{X}_U=\ind_U^G(1)$ as above. 

Note that $\Mod(G)$ is a Grothendieck category (the proof of part (iv) of \cite[Lem.~1]{Sch15} works in our generality). In particular $\Mod(G)$ has enough injective objects, and we choose an injective resolution $\Bbb{X}_U\rightarrow \II^\bb$. As in \cite[Sect.~3]{Sch15} this gives rise to a differential graded $k$-algebra
$\HH_U^\bb=\End_{\Mod(G)}^\bb(\II^\bb)^\op$ with cohomology algebra $h^*(\HH_U^\bb)$ isomorphic to $\Ext_{\Mod(G)}^*(\Bbb{X}_U, \Bbb{X}_U)^\op$.

We relate $D(G)$ to the derived category of dg modules $D(\HH_U^\bb)$ via the usual pair of adjoint functors
$$
\begin{tikzcd}
\frak{t}: D(\mathcal{H}_U^\bb) \arrow[bend left=35]{r}[name=F]{} & D(G): \frak{h}\arrow[bend left=35]{l}[name=U]{}
\end{tikzcd}
$$
defined as follows. To define $\frak{h}$ we first pick a fully faithful right adjoint ${\bf{i}}$ of the localization functor $q_G: K(G) \rightarrow D(G)$ with essential image $K_\text{inj}(G)$ 
the full subcategory of homotopically injective complexes. (Here $K(G)$ denotes the homotopy category of $\Mod(G)$.) Then for any object $V^\bb$ of $D(G)$ we let 
$$
\frak{h}(V^\bb):=q_\HH \big(\Hom_{\Mod(G)}^\bb(\II^\bb, {\bf{i}}V^\bb)\big)
$$
where $q_\HH: K(\mathcal{H}_U^\bb) \rightarrow D(\mathcal{H}_U^\bb)$ is the localization functor on the dg side. The left adjoint functor $\frak{t}$ is given by analogous considerations. We now pick a fully faithful left adjoint ${\bf{p}}$ of the localization functor $q_\HH$ whose essential image $K_\text{pro}(\mathcal{H}_U^\bb)$ 
is the full subcategory of homotopically projective dg modules, and let 
$$
\frak{t}(M^\bb):=q_G\big(\II^\bb \otimes_{\mathcal{H}_U^\bb} {\bf{p}}M^\bb \big)
$$
for any dg $\mathcal{H}_U^\bb$-module $M^\bb$. 

In the setting of \cite[Thm.~9]{Sch15} (where $G$ is a $p$-adic Lie group, $k$ is a field of characteristic $p$, and $U$ is a torsion-free pro-$p$ open subgroup) the two functors
$\frak{h}$ and $\frak{t}$ are mutually quasi-inverse equivalences of triangulated categories. In particular their restrictions identify the subcategory 
$\text{Thick}_{D(G)}(\Bbb{X}_U)$ with the category of perfect dg-modules $D(\mathcal{H}_U^\bb)^c=\text{Thick}_{D(\mathcal{H}_U^\bb)}(\mathcal{H}_U^\bb)$. The point of this subsection is the observation that the latter equivalence of thick {\it{subcategories}} continues to hold in the generality of a locally profinite group $G$ (and without restrictions on $U$). 
This is of course not original, and we only include it here for convenience. 

\begin{prop}\label{dgequiv}
The restrictions of $\frak{h}$ and $\frak{t}$ define mutually quasi-inverse equivalences of triangulated categories
$$
\text{Thick}_{D(G)}(\Bbb{X}_U) \overset{\sim}{\longrightarrow} D(\mathcal{H}_U^\bb)^c.
$$
\end{prop}

\begin{proof}
Note that the collections of objects 
$$
\{V^\bb: (\frak{t} \circ \frak{h})(V^\bb) \overset{\sim}{\longrightarrow} V^\bb\} \:\:\: \text{and} \:\:\:
\{M^\bb: M^\bb \overset{\sim}{\longrightarrow} (\frak{h}\circ \frak{t})(M^\bb)\}
$$
give thick subcategories of $D(G)$ and $D(\mathcal{H}_U^\bb)$ respectively since $\frak{t}$ and $\frak{h}$ both preserve \emph{finite} coproducts. 
As $\frak{h}(\II^\bb)\simeq \mathcal{H}_U^\bb$
and $\frak{t}(\mathcal{H}_U^\bb)\simeq \II^\bb$, the former subcategory contains $\II^\bb$, and the latter contains $\mathcal{H}_U^\bb$. Since the thick envelope of 
$\mathcal{H}_U^\bb$ is the subcategory of compact objects $D(\mathcal{H}_U^\bb)^c$ this shows that the restrictions of $\frak{h}$ and $\frak{t}$ yield an
equivalence of categories $\text{Thick}_{D(G)}(\Bbb{X}_U)=\text{Thick}_{D(G)}(\II^\bb) \overset{\sim}{\longrightarrow} D(\mathcal{H}_U^\bb)^c$.
\end{proof}

We emphasize that the previous argument does not require $\Bbb{X}_U$ to be in $D(G)^c$. 


\subsection{The center of a differential graded algebra}
 
We fix an arbitrary differential graded algebra over a field $k$, that is a cochain complex $A$ of $k$-vector spaces $\cdots \rightarrow A^r\overset{d^r}{\longrightarrow} A^{r+1} \rightarrow \cdots$ with $k$-bilinear maps $A^r \times A^s \rightarrow A^{r+s}$ such that $\bigoplus_{r \in \Z}A^r$ is an associative 
 $k$-algebra, and the Leibniz rule is satisfied. The algebra has a neutral element $1 \in A^0$ and $d(1)=0$. We denote its cohomology algebra by $h^*(A)$.

\begin{defn}
The \emph{naive center} of $A$ is the algebra $Z(A)=\bigoplus_{r \in \Z} Z^r(A)$ where the $r^\text{th}$ graded piece is the subspace
$$
Z^r(A)=\{a\in A^r: \text{$ab=(-1)^{rs}ba$ for all $b \in A^s$ and all $s$}\}.
$$
\end{defn}
 
\begin{rem}
The Leibniz rule shows that we have a subcomplex $\cdots \rightarrow Z^r(A)\overset{d^r}{\longrightarrow} Z^{r+1}(A) \rightarrow \cdots$ and thus $Z(A)$ 
is a graded-commutative differential graded $k$-subalgebra of $A$. We call this construction the \emph{naive} (or dg) center since $A \mapsto Z(A)$ does not respect quasi-isomorphisms in general. It will have more of an auxiliary role in what follows, and we will almost exclusively only consider the cohomology of $Z(A)$ in degree zero. 
 \end{rem}
 


 \begin{lem}\label{aM}
For every left differential graded $A$-module $M$, sending a cocycle $a \in Z^0(A)$ to the map $a_M: x \mapsto ax$ defines a homomorphism of $k$-algebras
$$
h^0(Z(A))\longrightarrow Z^0(\Hom_{D(A)}^*(M, M)).
$$
The map $a \mapsto (a_M)_{M \in D(A)}$ gives a homomorphism $h^0(Z(A)) \longrightarrow Z^0(D(A))$. 
\end{lem}

\begin{proof}
This is a straightforward argument and we leave the details to the reader.
\end{proof}

For $M=A$, viewed as a left module over $A$, we in particular get the homomorphism 
\begin{equation}\label{edge1}
h^0(Z(A))\longrightarrow Z^0(h^*(A))
\end{equation}
induced by the inclusion $Z(A) \subseteq A$.


\subsection{The characteristic homomorphism}
  
A perhaps better candidate for the center of $A$ is its zeroth Hochschild cohomology algebra, which we will denote by 
$HH^0(A)$. (Perhaps it would be more precise to write $HH^0(A, A)$ but since we will exclusively be considering the bimodule $A$ we find the duplicate notation unnecessary.) Hence an element of $HH^0(A)$ is a morphism 
$\xi: A \rightarrow A$ in the derived category of dg bimodules $D(A \otimes_k A^\op)$. For $M$ in 
$D(A)$ there is an associated morphism $\xi_{M}: M \rightarrow M$ in $D(A)$ obtained by tensoring $\xi$ with $\text{Id}_{M}$. This is sometimes referred to as the characteristic action, and it gives a homomorphism of $k$-algebras
\begin{align*}
\frak{c}: HH^0(A) & \longrightarrow Z^0(D(A)) \\
 \xi & \longmapsto (\xi_{M})_{M \in D(A)}.
\end{align*}
When $M=A$, viewed as a left dg module, $\xi_{A}$ is the image of $\xi$ in $D(A)$ obtained by forgetting the right module structure. Thus $\xi_A$ corresponds to the class of $\xi(1)$ via the identification
$$
\Hom_{D(A)}(A, A)=
\Hom_{K(A)}(A, A)\simeq h^0(A).
$$
Altogether, composing $\frak{c}$ with evaluation at $A$ gives the characteristic homomorphism (in degree zero)
\begin{align*}
\chi: HH^0(A) & \longrightarrow Z^0(h^*(A)) \\
 \xi & \longmapsto [\xi(1)]
\end{align*}
We are mainly interested in the image of $\chi$, which in general can be smaller than $Z^0(h^*(A))$.



We note that the homomorphism (\ref{edge1}), which is induced by the inclusion $Z(A) \subseteq A$, factors as $\chi \circ \varepsilon$ with 
$$
\varepsilon: h^0(Z(A))\longrightarrow HH^0(A)
$$
defined as follows: as we have seen, left multiplication by a cocycle $a \in Z^0(A)$ gives a map $a_A: A \rightarrow A$ of dg bimodules. The resulting morphism in $D(A \otimes_k A^\op)$ gives an element of $HH^0(A)$ which only depends on the class $[a]$ in 
$h^0(Z(A))$. With this definition it is immediately verified that $\chi \circ \varepsilon$ coincides with the map (\ref{edge1}) induced by inclusion.


We summarize our findings in the commutative diagram below. 
$$
\begin{tikzcd}
h^0(Z(A)) \arrow[r, "(\ref{edge1})"] \arrow[d, "\varepsilon"]
& Z^0(h^*(A)) \\
HH^0(A) \arrow[r,"\frak{c}"] \arrow[ur, "\chi"]
& Z^0(D(A)). \arrow[u, "\text{eval}_A"]
\end{tikzcd}
$$

Next we will apply these constructions to the dg Hecke algebra $\HH_U^\bb$.


\subsection{Constraints on the central Hecke operators $\tau_c$}

We return to the dg algebra $\HH_U^\bb$, in the setting of Section \ref{heckedga}, and we aim for a better understanding of the image of the characteristic homomorphism 
$$
\chi_U: HH^0(\HH_U^\bb) \longrightarrow Z^0(\Ext_{\Mod(G)}^*(\Bbb{X}_U, \Bbb{X}_U)).
$$
As a first observation, $\im(\chi_U)$ contains the Hecke operator $\tau_c=\text{char}_{UcU}$ defined by a $c \in Z(G)$.

\begin{lem}\label{constraint}
For all $c \in Z(G)$ we have $\tau_c \in \im(\chi_U)$.
\end{lem}

\begin{proof}
We begin by defining a homomorphism 
$$
\delta: Z(\Mod(G)) \longrightarrow h^0(Z(\HH_U^\bb)).
$$
Start with a tuple $t=(t_V)_{V\in \Mod(G)}$ in the center. To define $\delta(t)$ recall that elements of $Z^0(\HH_U^\bb)$ are tuples 
$a=(a_i)\in \prod_{i \in \Z}\End_{\Mod(G)}(\II^i)$ such that $ab=ba$ for all $b \in \HH_U^\bb$. Clearly $(t_{\II^i})$ is an example of such a tuple, since $t$ is in 
$Z(\Mod(G))$, and moreover it is a cocycle since
$$
d\big((t_{\II^i})\big)_j=d \circ t_{\II^j}-(-1)^0 \cdot t_{\II^{j+1}}\circ d=0.
$$
(On the right-hand side $d$ denotes the differential $d: \II^j \rightarrow \II^{j+1}$.) We let $\delta(t):=[(t_{\II^i})]$ be the class it defines in $h^0(Z(\HH_U^\bb))$. 

It now suffices to check that the diagram below commutes. (Note the close resemblance to (\ref{inidiag}); the next diagram is the analogue for $\mathcal{H}_U^\bb$, and we will soon upgrade it to (\ref{diag}).)
\begin{equation}\label{prediag}
\begin{tikzcd}
Z(\Mod(G)) \arrow[r, "\delta"] & h^0(Z(\HH_U^\bb)) \arrow[r, "\varepsilon"] & HH^0(\HH_U^\bb) \arrow[r, "\chi"] & Z^0(\Ext_{\Mod(G)}^*(\Bbb{X}_U, \Bbb{X}_U)) \\
k[Z(G)] \arrow[rrr, twoheadrightarrow, "\text{pr}"] \arrow[u, "\Phi"] & & & k[Z(G)/Z(G)\cap U].  \arrow[u, "c\mapsto \tau_c"]
\end{tikzcd}
\end{equation}
This is an easy diagram chase. Mapping $c \in Z(G)$ up via $\Phi$ gives the tuple of operators $(\Phi(c)_V)_{V \in \Mod(G)}$ where each component is given by the $c$-action.  
Thus $(\delta \circ \Phi)(c)$ is the cohomology class of $(\Phi(c)_{\II^i})$. As remarked in the previous subsection, $\chi\circ \varepsilon$ is the map on $h^0$ induced by the inclusion 
$Z(\HH_U^\bb)\subseteq \HH_U^\bb$. In particular $(\chi\circ \varepsilon \circ \delta \circ \Phi)(c)$ is the cohomology class defined by $(\Phi(c)_{\II^i})$ in $h^0(\HH_U^\bb)$. 

It remains to flesh out how $h^0(\HH_U^\bb)=\Hom_{K(G)}(\II^\bb, \II^\bb)^\op$ is identified with the algebra $\End_{\Mod(G)}(\Bbb{X}_U)^\op$. A morphism 
$\tau: \Bbb{X}_U\rightarrow \Bbb{X}_U$ in $\Mod(G)$ extends uniquely to a morphism $\widetilde{\tau}: \II^\bb \rightarrow \II^\bb$ in $K(G)$, and $\tau \mapsto \widetilde{\tau}$ gives an isomorphism $\End_{\Mod(G)}(\Bbb{X}_U)^\op \overset{\sim}{\longrightarrow} h^0(\HH_U^\bb)$ since $\II^\bb$ is homotopically injective. The $\tau$ corresponding to $\tau_c$ is 
$\Phi(c)_{\Bbb{X}_U}$, and in this case (by uniqueness of the extension to $\II^\bb$) the associated $\widetilde{\tau}$ can be represented by the morphism of complexes 
$\II^\bb \rightarrow \II^\bb$ which is $\Phi(c)_{\II^i}$ in degree $i$.
\end{proof}

Consequently the map $c \mapsto \tau_c$ gives a homomorphism into the image of the characteristic homomorphism, say
$$
\gamma: k[Z(G)/Z(G)\cap U] \longrightarrow \im(\chi_U).
$$
This is obviously injective since the target is contained in the Hecke algebra, and $c \mapsto \tau_c$ is injective as a morphism into 
$\End_{\Mod(G)}(\Bbb{X}_U)^\op$. Oftentimes $\{\tau_c\}_{c\in Z(G)}$ is known to span $Z^0(\Ext_{\Mod(G)}^*(\Bbb{X}_U, \Bbb{X}_U))$, in which case $\chi_U$ is of course surjective. This happens for split $p$-adic reductive groups for example, with $U$ being pro-$p$ Iwahori, as we explained after Corollary \ref{maincor} in Section \ref{padiclie}.

We proceed to study the image of the map $\gamma_\text{red}$ induced by $\gamma$ on nilreductions, using Theorem \ref{mainr} (and Proposition \ref{dgequiv}). 

\subsection{The image of $\chi_U$ modulo nilpotent elements}

We now establish the main result of this section, Theorem \ref{imchar}, whose proof relies heavily on Theorem \ref{mainr}. First we note that the equivalence in Proposition \ref{dgequiv} induces an isomorphism of algebras $\frak{B}$ fitting in the diagram 
\begin{equation}\label{diag}
\begin{tikzcd}
h^0(Z(\mathcal{H}_U^\bb)) \arrow[rr, "\frak{c}\circ \varepsilon"]& & Z^0(D(\mathcal{H}_U^\bb)^c) \arrow[d, "\frak{B}", "\simeq"'] \arrow[r, "\text{eval}_{\HH_U^\bb}"] & \im(\chi_U)  \\
Z(\text{Mod}(G)) \arrow[rr, "s"] \arrow[u, "\delta"]& &  Z^0(\text{Thick}_{D(G)}(\Bbb{X}_U)) &  \\
k[Z(G)]\arrow[rrr, twoheadrightarrow, "\text{pr}"] \arrow[u, hook, "\Phi"] & & & k[Z(G)/Z(G) \cap U]. \arrow[uu, hook, "\gamma"] 
\end{tikzcd}
\end{equation}
It is given by sending a tuple $(\alpha_{M^\bb})$ to the tuple $\frak{B}((\alpha_{M^\bb}))$ whose $V^\bb$-component corresponds to $\frak{t}(\alpha_{\frak{h}(V^\bb)})$ via the adjunction counit $(\frak{t}\circ \frak{h})(V^\bb) \overset{\sim}{\longrightarrow} V^\bb$. It will be a non-trivial key input that the previous diagram commutes. 

The outer rectangle of (\ref{diag}) is essentially the diagram (\ref{prediag}) in the proof of Lemma \ref{constraint}, so (\ref{diag}) is a refinement of (\ref{prediag}) and we are left with showing the inner rectangle of (\ref{diag}) commutes.

\begin{lem}\label{commdiag}
The diagram (\ref{diag}) is commutative. 
\end{lem}

\begin{proof}
We start with the upper left rectangle and chase an element $t=(t_V)_{V \in \Mod(G)}$ in $Z(\text{Mod}(G))$ around it. First, going across, remember that the section $s$ is defined in such a way that 
$s(t)$ is the tuple arising from the morphisms of complexes $t_{V^\bb}: V^\bb \rightarrow V^\bb$ with $t_{V^i}$ in degree $i$. 

On the other hand, taking the detour around the rectangle first yields the class $\delta(t)=[(t_{\II^i})]$. We let 
$a:=(t_{\mathcal{I}^i}) \in \prod_{i \in \Z}\Hom_{\text{Mod}(G)}(\mathcal{I}^i,\mathcal{I}^i)$ for the duration of this proof. Unwinding the definitions of $\varepsilon$ and $\frak{c}$ we see that $(\frak{c}\circ \varepsilon \circ \delta)(t)$ is left multiplication by $a$, denoted $a_{M^\bb}: M^\bb \rightarrow M^\bb$, on every perfect dg module $M^\bb$. We must compare the tuple $\frak{B}((a_{M^\bb}))$ to $s(t)$. The $V^\bb$-component of $\frak{B}((a_{M^\bb}))$ corresponds to $\frak{t}(a_{\frak{h}(V^\bb)})$ via the adjunction counit
$\frak{t}\circ \frak{h} \overset{\sim}{\longrightarrow} \text{Id}$. To identify this with $t_{V^\bb}$ we unpack the definitions. We begin with the dg module 
$\Hom_{\Mod(G)}^\bb(\mathcal{I}^\bb, {\bf{i}}V^\bb)$ and the map $(\frak{t} \circ \frak{h})(V^\bb)\rightarrow V^\bb$ which is $q_G$ applied to the following composition in the homotopy category $K(G)$, followed by composition with $q_G{\bf{i}}\rightarrow \text{Id}$, 
$$
\mathcal{I}^\bb\otimes_{\mathcal{H}_U^\bb} {\bf{p}}q_\HH\Hom_{\Mod_k(G)}^\bb(\mathcal{I}^\bb, {\bf{i}}V^\bb) \rightarrow 
\mathcal{I}^\bb\otimes_{\mathcal{H}_U^\bb} \Hom_{\Mod_k(G)}^\bb(\mathcal{I}^\bb, {\bf{i}}V^\bb) \rightarrow {\bf{i}}V^\bb.
$$
To simplify the notation we will assume that $V^\bb$ is homotopically injective and henceforth identify it with ${\bf{i}}V^\bb$. Now look at the diagram
$$
\begin{tikzcd}
\mathcal{I}^\bb\otimes_{\mathcal{H}_U^\bb} {\bf{p}}q_\HH\Hom_{\Mod(G)}^\bb(\mathcal{I}^\bb, V^\bb) \arrow[r] \arrow[d, "1\otimes a=a\otimes 1"] & \mathcal{I}^\bb\otimes_{\mathcal{H}_U^\bb} \Hom_{\Mod(G)}^\bb(\mathcal{I}^\bb, V^\bb) \arrow[d, "1\otimes a=a\otimes 1"] \arrow[r]& V^\bb \arrow[d, "t_{V^\bb}"] \\
\mathcal{I}^\bb\otimes_{\mathcal{H}_U^\bb} {\bf{p}}q_H\Hom_{\Mod(G)}^\bb(\mathcal{I}^\bb, V^\bb) \arrow[r] & \mathcal{I}^\bb\otimes_{\mathcal{H}_U^\bb} \Hom_{\Mod(G)}^\bb(\mathcal{I}^\bb, V^\bb) \arrow[r] & V^\bb
\end{tikzcd}
$$
The leftmost vertical arrow induces $\frak{t}(a_{\frak{h}(V^\bb)})$ when applying $q_G$. It therefore suffices to check that the rightmost square commutes. This comes down to verifying that 
$$
\begin{tikzcd}
\mathcal{I}^i\otimes_k \Hom_{\Mod(G)}^{j}(\mathcal{I}^\bb, V^\bb) \arrow[r] \arrow[d, "a\otimes 1"]
& V^{i+j} \arrow[d, "t_{V^{i+j}}"] \\
\mathcal{I}^i\otimes_k \Hom_{\Mod(G)}^j(\mathcal{I}^\bb, V^\bb) \arrow[r]
& V^{i+j}
\end{tikzcd}
$$
commutes for all $i$ and $j$, but this is straightforward. Indeed, for $x \in \mathcal{I}^i$ and $\phi \in \Hom_{\Mod(G)}^{j}(\mathcal{I}^\bb, V^\bb)$,
$$
t_{V^{i+j}}(\phi_i(x))=\phi_i(t_{\mathcal{I}^i}(x))=\phi_i(xa).
$$
This shows the inner rectangle commutes.

The perimetral rectangle was already shown to be commutative in the course of the proof of Lemma \ref{constraint}, where it was verified that $(\chi \circ \varepsilon \circ \delta \circ \Phi)(c)$ agrees with $\tau_c$. 
\end{proof}




We can now finally establish the main result of this section. 

\medskip

\noindent {\it{Proof of Theorem \ref{imchar}}}.
As noted earlier, $\gamma: c \mapsto \tau_c$ injects $k[Z(G)/Z(G)\cap U]$ into the Hecke algebra, and as a result at least we have an injective homomorphism 
$$
\gamma_\text{red}: k[Z(G)/Z(G)\cap U]_\text{red} \hookrightarrow \im(\chi_U)_\text{red}.
$$
In order to show it is also surjective, suppose $y \in \im(\chi_U)$ and write $y=\chi(\xi)$ for some $\xi \in HH^0(\HH_U^\bb)$. Recall that $\chi_U=\text{eval}_{\HH_U^\bb}\circ \frak{c}$. Now consider $\frak{c}(\xi)$ as an element of 
$Z^0(D(\HH_U^\bb)^c)$ by restricting it to just the compact objects. Via Proposition \ref{dgequiv} we get a corresponding element 
$\frak{B}(\frak{c}(\xi)) \in Z^0(\text{Thick}_{D(G)}(\Bbb{X}_U))$. By Theorem \ref{mainr} there is an $x \in k[Z(G)]$ such that $(s\circ \Phi)(x)$ agrees with 
$\frak{B}(\frak{c}(\xi))$ up to something locally nilpotent. Since the diagram (\ref{diag}) is commutative, we deduce that 
$(\frak{c}\circ \varepsilon \circ \delta \circ \Phi)(x)$ agrees with $\frak{c}(\xi)$ up to something locally nilpotent. Finally, evaluation at $\HH_U^\bb$ allows us to conclude that 
$\gamma(\text{pr}(x))$ agrees with $y=\chi(\xi)$ up to something nilpotent. Therefore the image of $y$ in the nilreduction lies in the image of $\gamma_\text{red}$ as claimed.

\section{Bounded variants of the graded center}

If $G$ is a locally pro-$p$ group, satisfying the usual centralizer hypothesis, we will identify $Z^0(D^b(G))_\text{lred}$ with the nilreduction
$\widehat{k[Z(G)]}_\text{nil}$ in Theorem \ref{bounded} below. Here $(-)_\text{nil}$ signifies the 'actual' nilreduction, as opposed to the topological nilreduction $(-)_\text{red}$ considered earlier. We also verify that the kernel of the restriction map $Z^*(D^+(G))\rightarrow Z^*(D^b(G))$ has square zero. 

When $G$ happens to be a $d$-dimensional $p$-adic Lie group we can say more. In this case $\Mod(G)$ satisfies Roos's condition AB4*-$d$, i.e. the derived products 
$\Pi^{(i)}$ vanish for $i>d$, which implies that every object $V^\bb$ of $D(G)$ is a homotopy limit of its truncations $\tau_{\geq -n}V^\bb$. Using this, 
the kernel of the restriction map $Z^*(D(G))\rightarrow Z^*(D^+(G))$ is easily seen to also have square zero.

\subsection{Some background in homological algebra}

A Grothendieck category $\mathcal{C}$ is automatically AB3* but not necessarily AB4*. That is $\mathcal{C}$ has products, and even limits by \cite[Cor.~X.4.4]{Ste75}, but taking products is not exact in general. Products can be made explicit using the Gabriel-Popescu theorem, for which we follow \cite[Thm.~19.14.3]{Sta}.
Namely, if $Y$ is a generator of $\mathcal{C}$, and $E=\End_{\mathcal{C}}(Y)^\text{op}$, then $\Hom_\mathcal{C}(Y,-)$ is a fully faithful functor $f: \mathcal{C}\rightarrow \Mod(E)$ with an exact left adjoint, say $t: \Mod(E)\rightarrow \mathcal{C}$. Note that $t \circ f \overset{\sim}{\longrightarrow} \text{Id}$. With this notation, products in $\mathcal{C}$ can be expressed as follows. If $(V_\alpha)_{\alpha \in A}$ is a set-indexed collection of objects in $\mathcal{C}$ it is straightforward to verify that 
$$
\prod_{\alpha \in A} V_\alpha=t\big(\prod_{\alpha \in A}f(V_\alpha)\big)=t\big(\prod_{\alpha \in A}\Hom_\mathcal{C}(Y,V_\alpha)\big)
$$
where the product on the right-hand side is the one in $\Mod(E)$. To see this note that $f$ commutes with limits because it has a left adjoint, namely $t$. Now apply $t$ to both sides of $f(\prod_{\alpha \in A} V_\alpha)=\prod_{\alpha \in A} f(V_\alpha)$.

Because $t$ is exact, and products in $\Mod(E)$ \emph{are} exact, we immediately also get an explicit formula for the $i^\text{th}$ derived product, i.e.
$$
{\prod_{\alpha \in A}}^{(i)} V_\alpha=t\big(\prod_{\alpha \in A}R^if(V_\alpha)\big)=t\big(\prod_{\alpha \in A}\Ext_\mathcal{C}^i(Y,V_\alpha)\big).
$$
In particular we deduce that if $\Ext_\mathcal{C}^i(Y,-)=0$ for $i>d$ then $\mathcal{C}$ satisfies Roos's condition AB4*-$d$ from \cite[Def.~1.1, p.~66]{Roo08}, that is to say ${\prod_{\alpha \in A}}^{(i)}$ is zero
for $i>d$, so at least taking products has finite cohomological amplitude when such a $Y$ exists. (We should note that the above formula for ${\prod_{\alpha \in A}}^{(i)}$ also 
appears at the very bottom of \cite[p.~67]{Roo08} in a different notation.)

The reason this is important to us is the following result, notably part two. 

\begin{lem}
Assume $\mathcal{C}$ is Grothendieck, and let $V^\bb$ be a complex representing an object of $D(\mathcal{C})$. 
\begin{itemize}
\item[(1)] $V^\bb$ is a homotopy colimit of $\tau_{\leq 0}V^\bb \rightarrow \tau_{\leq 1}V^\bb \rightarrow \cdots$, i.e. there is an isomorphism
$$
\underset{n\geq 0}{\text{hocolim}}\: \tau_{\leq n}V^\bb \overset{\sim}{\longrightarrow} V^\bb.
$$
(This is true more generally for AB4 categories, which means coproducts exist and are exact.)
\item[(2)] $V^\bb$ is a homotopy limit of $\tau_{\geq 0}V^\bb \leftarrow \tau_{\geq -1}V^\bb \leftarrow \cdots$, i.e. there is an isomorphism
$$
V^\bb \overset{\sim}{\longrightarrow} \underset{n\geq 0}{\text{holim}}\: \tau_{\geq -n}V^\bb 
$$
provided $\mathcal{C}$ is AB4* or even just AB4*-$d$ for some $d \in \Z_{\geq 0}$. 
\end{itemize}
\end{lem}

\begin{rem}
We note that $D(\mathcal{C})$ has coproducts since $\mathcal{C}$ is AB4. When $\mathcal{C}$ is \emph{not} Grothendieck $D(\mathcal{C})$ may not have products. Nevertheless, if $\mathcal{C}$ at least has enough injectives we can take the product of objects in $D^+(\mathcal{C})$. 
We refer to \cite[Lem.~13.33.5, Lem.~13.34.2, Lem.~13.34.3]{Sta} for more details on the existence of coproducts and products in $D(\mathcal{C})$ for arbitrary $\mathcal{C}$. \
Since we are assuming our $\mathcal{C}$ \emph{is} Grothendieck this will not be an issue; see Appendix A. 
\end{rem}

\begin{proof}
Part (1) is well-known, and so is the dual assertion in part (2) for AB4* categories (see \cite[Rem.~2.3, p.~214]{BN93} or the proof of \cite[Lem.~13.34.7]{Sta} for example). 
The generalization to AB4*-$d$ categories is stated as \cite[Thm.~1.3]{HX09}. We give a (in our opinion) simpler proof below. 

We fix an $n_0 \in \Z_{\geq 0}$ once and for all, and consider the inverse system $(\tau_{\geq -n}V^\bb)_{n \geq n_0}$ which maps to the 'constant' inverse system
$(\tau_{\geq -n_0}V^\bb)_{n \geq n_0}$ via the transition maps $\tau_{\geq -n}V^\bb \rightarrow \tau_{\geq -n_0}V^\bb$ for $n \geq n_0$. This gives rise to the following commutative diagram in $D(\mathcal{C})$ whose rows are distinguished triangles, 
$$
\begin{tikzcd}
\underset{n\geq n_0}{\text{holim}}\: \tau_{\geq -n}V^\bb  \arrow[r] \arrow[d, dashrightarrow, "\varphi"] & \prod_{n \geq n_0} \tau_{\geq -n}V^\bb \arrow[r, "1-\text{sh}"] \arrow[d]
& \prod_{n \geq n_0} \tau_{\geq -n}V^\bb \arrow[d] \arrow[r] & \Sigma \big(\underset{n\geq n_0}{\text{holim}}\: \tau_{\geq -n}V^\bb \big) \arrow[d, "\Sigma(\varphi)"]\\
\tau_{\geq -n_0}V^\bb \arrow[r] & \prod_{n \geq n_0} \tau_{\geq -n_0}V^\bb \arrow[r, "1-\text{sh}"]
& \prod_{n \geq n_0} \tau_{\geq -n_0}V^\bb \arrow[r] & \Sigma \big(\tau_{\geq -n_0}V^\bb\big)
\end{tikzcd}
$$
where $\text{sh}$ stands for the usual shift morphism. Passing to the long exact sequences of cohomology gives commutative diagrams in $\mathcal{C}$, 
$$
\begin{tikzcd}
h^i(\underset{n\geq n_0}{\text{holim}}\: \tau_{\geq -n}V^\bb)  \arrow[r] \arrow[d, "h^i(\varphi)"] & h^i(\prod_{n \geq n_0} \tau_{\geq -n}V^\bb) \arrow[r, "1-\text{sh}"] \arrow[d]
& h^i(\prod_{n \geq n_0} \tau_{\geq -n}V^\bb) \arrow[d] \arrow[r] & h^{i+1}(\underset{n\geq n_0}{\text{holim}}\: \tau_{\geq -n}V^\bb) \arrow[d, "h^{i+1}(\varphi)"]\\
h^i(\tau_{\geq -n_0}V^\bb) \arrow[r] & h^i(\prod_{n \geq n_0} \tau_{\geq -n_0}V^\bb) \arrow[r, "1-\text{sh}"]
& h^i(\prod_{n \geq n_0} \tau_{\geq -n_0}V^\bb) \arrow[r] & h^{i+1}(\tau_{\geq -n_0}V^\bb).
\end{tikzcd}
$$
We claim the middle two vertical maps are isomorphisms for $i>d-n_0$. To see this we use the following 'hypercohomology' spectral sequence: Let $(W_n^\bb)_{n \in I}$ be a set-indexed family of objects in $D(\mathcal{C})$. Then, as shown in Appendix A, their product $\prod_{n \in I} W_n^\bb$ exists in $D(\mathcal{C})$ and 
there is a convergent spectral sequence
$$
E_2^{r,s}={\prod_{n \in I}}^{(r)} h^s(W_n^\bb) \Longrightarrow h^{r+s}(\prod_{n \in I} W_n^\bb). 
$$
It is concentrated in the vertical strip $0\leq r\leq d$ as $\mathcal{C}$ is AB4*-$d$. Applying this to $(\tau_{\geq -n}V^\bb)_{n \geq n_0}$ and $(\tau_{\geq -n_0}V^\bb)_{n \geq n_0}$
it suffices to show that on the $E_2$-pages the maps 
$$
{\prod_{n \geq n_0}}^{(r)} h^s(\tau_{\geq -n}V^\bb) \longrightarrow {\prod_{n \geq n_0}}^{(r)} h^s(\tau_{\geq -n_0}V^\bb)
$$
are isomorphisms for $r+s>d-n_0$ (see the 'comparison' theorem \cite[Thm.~5.2.12, p.~126]{Wei94} for example). This is straightforward. Either $r>d$ or $s>-n_0$. If $r>d$ the source and the target are both trivial since ${\prod_{n \geq n_0}}^{(r)}=0$. On the other hand, if $s>-n_0$ we have the isomorphisms
$$
h^s(V^\bb)\simeq h^s(\tau_{\geq -n}V^\bb) \longrightarrow h^s(\tau_{\geq -n_0}V^\bb) \simeq h^s(V^\bb)
$$
for $n \geq n_0$. This proves the claim about the middle two vertical maps being isomorphisms.

If in fact $i>d-n_0+1$ the same conclusion also applies to $i-1$, in addition to $i$, and the five lemma lets us conclude that $h^i(\varphi)$ is an isomorphism
$$
h^i(\underset{n\geq n_0}{\text{holim}}\: \tau_{\geq -n}V^\bb) \overset{\sim}{\longrightarrow} h^i(\tau_{\geq -n_0}V^\bb) \simeq h^i(V^\bb).
$$
However, the left-hand side is independent of $n_0$ up to isomorphism (see \cite[Rem.~13.34.4, Lem.~15.88.6]{Sta}). Since $n_0$ was arbitrary the map $h^i(V^\bb) \longrightarrow h^i(\underset{n\geq 0}{\text{holim}}\: \tau_{\geq -n}V^\bb)$ must be an isomorphism for all $i$. 
\end{proof}

As an application, suppose $t \in Z^*(D(\mathcal{C}))$ is homogeneous and $t$ vanishes on all the truncations $\tau_{\leq n}V^\bb$. Then $t$ is also zero on 
$\bigoplus_{n\geq 0}\tau_{\leq n}V^\bb$ and therefore, by our general remarks, $t^2$ vanishes on $\underset{n\geq 0}{\text{hocolim}}\: \tau_{\leq n}V^\bb$. 
(Take $\mathcal{X}$ in Lemma \ref{nilp} to be the single element $\bigoplus_{n\geq 0}\tau_{\leq n}V^\bb$ and observe that the homotopy colimit lies in 
$\langle \mathcal{X} \rangle_2$ by its defining triangle.) Similarly for homotopy limits. 

The upshot of this is the next general observation.

\begin{prop}\label{sq0}
Let $\mathcal{C}$ be a Grothendieck category. Then the following holds. 
\begin{itemize}
\item[(1)] The kernel of the restriction map $Z^*(D^+(\mathcal{C}))\longrightarrow Z^*(D^b(\mathcal{C}))$ has square zero;
\item[(2)] The kernel of $Z^*(D(\mathcal{C}))\longrightarrow Z^*(D^+(\mathcal{C}))$ has square zero if $\mathcal{C}$ is AB4*-$d$ for some $d \in \Z_{\geq 0}$. 
\end{itemize}
\end{prop}

\begin{proof}
In part (1) suppose $t \in Z^*(D^+(\mathcal{C}))$ vanishes on all objects of $D^b(\mathcal{C})$. If $V^\bb$ is in $D^+(\mathcal{C})$ all its truncations 
$\tau_{\leq n}V^\bb$ belong to $D^b(\mathcal{C})$. Thus $t$ is zero on their coproduct and consequently $t^2$ is zero on their homotopy colimit which is $V^\bb$. 
More generally, if $\sigma, \tau$ both lie in the kernel then, as we have seen in Lemma \ref{nilp}, $\sigma\tau$ vanishes on the homotopy colimit $V^\bb$ as well, so indeed the kernel squares to zero.  

For part (2) assume $t \in Z^*(D(\mathcal{C}))$ is zero on all objects of $D^+(G)$. For any $V^\bb$ the truncations $\tau_{\geq -n}V^\bb$ now lie in $D^+(\mathcal{C})$ and thus 
$t$ is zero on their product. As in the previous paragraph we conclude that $t^2$ is zero on $V^\bb$.  
\end{proof}

\begin{rem}
We emphasize that this does {\emph{not}} guarantee that the induced map $Z^*(D^+(\mathcal{C}))_\text{lred}\longrightarrow Z^*(D^b(\mathcal{C}))_\text{lred}$, for example, is injective: If $t \in Z^*(D^+(\mathcal{C}))$ is nilpotent on all objects of $D^b(\mathcal{C})$ the nilpotency index on $\tau_{\leq n}V^\bb$ could very well depend on $n$, in which case we would not necessarily be able to conclude that $t$ is nilpotent on $\bigoplus_{n\geq 0}\tau_{\leq n}V^\bb$ (and hence on the homotopy colimit $V^\bb$). 

On the other hand the induced map between the 'actual' nil-reductions (which means modding out the homogeneous nilpotent elements) \emph{is} injective by Proposition \ref{sq0}.
\end{rem}

\begin{rem}\label{nonneg}
We note that $Z^*(D^b(\mathcal{C}))$ is known to be non-negatively graded, that is $Z^r(D^b(\mathcal{C}))=0$ for all $r<0$, by \cite[Cor.~2.8, Rem.~2.9, p.~449]{KY11} which applies to any abelian category $\mathcal{C}$ with enough injectives (or enough projectives). In all candor, \cite[Cor.~2.8]{KY11} assumes $\mathcal{C}$ has enough projectives, but the 'injective' counterpart follows immediately by replacing $\mathcal{C}$ with $\mathcal{C}^\text{op}$ and observing that $Z^*(\mathcal{T}^\text{op})\simeq Z^*(\mathcal{T})$ for any triangulated category $\mathcal{T}$, e.g. $D^b(\mathcal{C})$.

For $r<0$, we deduce from Proposition \ref{sq0} that all elements of $Z^r(D^+(\mathcal{C}))$ have square zero, and provided $\mathcal{C}$ is AB4*-$d$ for some $d$, all elements $t \in Z^r(D(\mathcal{C}))$ have fourth power $t^4=0$. In particular, for negative $r$,  the elements of $Z^r(D(\mathcal{C}))$ are nilpotent (in the AB4*-$d$
 case) and not just locally nilpotent.
\end{rem}

In general we have the split surjection $\pi^b: Z^0(D^b(\mathcal{C}))\twoheadrightarrow Z(\mathcal{C})$ obtained by viewing $\mathcal{C}$ as a full subcategory of $D^b(\mathcal{C})$
in the usual fashion (as complexes concentrated in degree zero). 

\begin{prop}\label{DbC}
The kernel of $\pi^b$ consists of locally nilpotent elements, and $\pi^b$ induces an isomorphism
$$
Z^0(D^b(\mathcal{C}))_\text{lred}\overset{\sim}{\longrightarrow} Z(\mathcal{C})_\text{lred}.
$$
(Here $\mathcal{C}$ can be any abelian category.)
\end{prop}

\begin{proof}
It is well-known that any object of $D^b(\mathcal{C})$ can be represented, up to isomorphism, by a bounded complex. See part (3) of \cite[Lem.~13.11.5]{Sta} for instance.
By induction on the amplitude of this complex, one easily verifies that $D^b(\mathcal{C})$ is the smallest strictly full triangulated subcategory of $D(\mathcal{C})$ containing $\mathcal{C}$. Hence
$D^b(\mathcal{C})=\text{Thick}_{D(\mathcal{C})}(\mathcal{C})$ since $D^b(\mathcal{C})$ is clearly thick. Now the first claim follows from Lemma \ref{nilp} by taking 
$\mathcal{T}=D^b(\mathcal{C})$ and $\mathcal{X}=\mathcal{C}$.

A splitting $s$ of $\pi^b$ is defined by taking $(t_V)_{V \in \mathcal{C}}$ to the collection of morphisms of complexes 
$(t_{V^\bb})_{V^\bb \in D^b(\mathcal{C})}$ where $t_{V^\bb}$ is given by $t_{V^i}$ in degree $i$.

Obviously $\pi^b$ sends locally nilpotent elements to locally nilpotent elements, and therefore induces a map 
$Z^0(D^b(\mathcal{C}))_\text{lred}\rightarrow Z(\mathcal{C})_\text{lred}$. On the other hand, $s$ also preserves locally nilpotent elements due to the boundedness assumption; 
$V^i=0$ for $|i|>\!\!>0$ so given $V^\bb$ there is a large enough $N$ such that $t_{V^i}^N=0$ for all $i$, i.e. $t_{V^\bb}^N=0$. Thus $s$ induces a map on $(-)_\text{lred}$ in the other direction, $Z(\mathcal{C})_\text{lred}\rightarrow Z^0(D^b(\mathcal{C}))_\text{lred}$. Since $\pi^b \circ s=\text{Id}$ and $t-(s \circ \pi^b)(t)$ is in $\ker(\pi^b)$, and is therefore locally nilpotent, for all $t \in Z^0(D^b(\mathcal{C}))$, we conclude that the two maps induced by $\pi^b$ and $s$ on $(-)_\text{lred}$ are mutually inverse isomorphisms. 
\end{proof}

\begin{rem}\label{hered}
When $\mathcal{C}$ is hereditary, i.e. $\Ext_\mathcal{C}^i(V,W)=0$ for $i>1$ and all $V,W \in \mathcal{C}$, Proposition \ref{DbC} holds without passing to
$(-)_\text{lred}$. In other words the splitting $s$ is an inverse of $\pi^b$, resulting in the isomorphism 
$Z^0(D^b(\mathcal{C}))\overset{\sim}{\longrightarrow} Z(\mathcal{C})$. Indeed, to see that $(s \circ \pi^b)(t)=t$ we compare the $V^\bb$-components of both sides.  
By \cite[Cor.~13.1.20, p.~324]{KS06} there is a non-canonical isomorphism $V^\bb \simeq \bigoplus_{i \in \Z} \Sigma^{-i} \big(h^i(V^\bb)\big)$. Thus we may (and will) assume that $V^\bb=\bigoplus_{i \in \Z} \Sigma^{-i}\big(V^i\big)$ is a complex with zero-differentials. Since $t$ is functorial, and it commutes with $\Sigma$, we find that 
$t_{V^\bb}=\bigoplus_{i \in \Z} \Sigma^{-i}\big(t_{V^i}\big)$. Unwinding the definitions of $s$ and $\pi^b$, the right-hand side is exactly $(s \circ \pi^b)(t)_{V^\bb}$.
\end{rem}


\subsection{On the center of $D^b(G)$ and other variants}

We first stress that the results from the previous section apply to $D(G)$ when $G$ is a $p$-adic Lie group, indeed we have:

\begin{lem}\label{ab4}
Let $G$ be a $p$-adic Lie group of dimension $d$ over $\Q_p$. Then $\Mod(G)$ is AB4*-$d$.
\end{lem}

\begin{proof}
Pick a uniform pro-$p$ open subgroup $U$ and consider the generator $Y:=\bigoplus_{U'\subset U} \Bbb{X}_{U'}$. Since each $U'$ in particular is torsion-free we find that
$$
\Ext_{\Mod(G)}^i(Y,V)=\prod_{U' \subset U} \Ext_{\Mod(G)}^i( \Bbb{X}_{U'},V)=\prod_{U' \subset U} H^i(U',V)=0
$$ 
for any $V$ in $\Mod(G)$ as long as $i>d$ (using that $k[\![U']\!]$ has global dimension $d$ by work of Lazard). 
\end{proof}

Let us summarize our findings after introducing a bit of notation. Since we have used $(-)_\text{red}$ for the quotient by the topologically nilpotent elements, we will use the notation 
$(-)_\text{nil}$ for the quotient by the 'actual' nilpotent elements. We remind the reader that $s$ denotes the splitting of $\pi^b$ introduced in the proof of 
Proposition \ref{DbC}.

\begin{thm}\label{bounded}
Let $G$ be a locally pro-$p$ group without proper open centralizers, and let $k$ be a field of characteristic $p$. Then we have the following.
\begin{itemize}
\item[(1)] The map $Z(G) \ni c \mapsto (s \circ \Phi)(c)$ induces an isomorphism of $k$-algebras
$$
\widehat{k[Z(G)]}_\text{nil} \overset{\sim}{\longrightarrow} Z^0(D^b(G))_\text{lred}.
$$
\item[(2)] The kernel of the restriction map $Z^*(D^+(G))\longrightarrow Z^*(D^b(G))$ has square zero;
\item[(3)] The kernel of $Z^*(D(G))\longrightarrow Z^*(D^+(G))$ has square zero if $G$ is a $p$-adic Lie group. 
\end{itemize}
\end{thm} 

\begin{proof}
We know from Proposition \ref{DbC} that the restriction map $\pi^b$ induces an isomorphism
$$
Z^0(D^b(G))_\text{lred}\overset{\sim}{\longrightarrow} Z(\Mod(G))_\text{lred}.
$$
We claim that $t=(t_V)_{V \in \Mod(G)}\in Z(\Mod(G))$ is locally nilpotent (if and) only if $t$ is nilpotent. Indeed $t_C$ must be nilpotent, where $C:=C_c^\infty(G,k)$ is viewed as a $G$-representation via left translations, say. Write $t=\Phi(x)$ with $x \in \widehat{k[Z(G)]}$ where $\Phi: \widehat{k[Z(G)]} \overset{\sim}{\longrightarrow} Z(\Mod(G))$ is the isomorphism in  \cite[Thm.~6.10, p.~15]{AS23}. Since $t_C=\Phi_C(x)$ is nilpotent, and $\Phi_C$ is injective by part (1) of \cite[Prop.~3.8, p.~4]{AS23}, we deduce that $x$ and therefore $t$ must be nilpotent. 

In conclusion $\Phi$ gives rise to the isomorphism
$$
\widehat{k[Z(G)]}_\text{nil} \underset{\Phi}{\overset{\sim}{\longrightarrow}} Z(\Mod(G))_\text{nil}=\joinrel=Z(\Mod(G))_\text{lred}
\underset{s}{\overset{\sim}{\longrightarrow}} Z^0(D^b(G))_\text{lred}
$$
in part (1). Parts (2) and (3) follow immediately from Proposition \ref{DbC} and Lemma \ref{ab4}.
\end{proof}

We remind the reader that $Z^r(D^b(G))=0$ for all $r<0$ by Remark \ref{nonneg} applied to $\mathcal{C}=\Mod(G)$. 


\section{Remarks in the locally coherent case}

In this section we assume throughout that $G$ is a locally pro-$p$ group for which $\Mod(G)$ is locally coherent. By definition this means the subcategory of finitely presented objects 
$\Mod(G)^\text{fp}$ is abelian, which in our context essentially amounts to $\Mod(G)^\text{fp}$ being closed under taking kernels. We remind the reader that, following \cite[p.~122]{Ste75}, an object $V$ is said to be finitely presented if (i) $V$ is finitely generated; and (ii) every epimorphism $W \twoheadrightarrow V$, from a finitely generated $W$, has a finitely generated kernel. As $\Mod(G)$ is locally finitely generated, this is equivalent to $\Hom_{\Mod(G)}(V,-)$ commuting with filtered colimits (see \cite[Prop.~3.4, p.~122]{Ste75}). 

\begin{exmps}
If $G$ is compact, $\Mod(G)$ is even locally Noetherian (the $\Bbb{X}_{U'}$ give a family of finite-dimensional generators). By \cite[Prop.~2.1.4, p.~6]{AS24}, $\Mod(G)$ is also locally Noetherian for $G=\SL_2(\Q_p)$. When $G=\SL_2(\frak{F})$, for a general finite extension $\frak{F}/\Q_p$, $\Mod(G)$ is locally coherent by \cite[Thm.~1.2, p.~144]{Sho20}
(a different proof was presented in \cite{Sch26} which works more generally for all amalgams).

Some authors expect $\Mod(G)$ to be locally coherent for $G=\GL_n(\frak{F})$ in general; see \cite[Conj.~6.1.4]{EGH25} for example (which is formulated for more general coefficient rings $\mathcal{O}$). 
\end{exmps}

In modular representation theory it is common to replace $D(G)$ with the homotopy category of injectives $K(\text{Inj}(G))$. Under our standing 'local coherence' hypothesis (in this section) it is known that $K(\text{Inj}(G))$ is compactly generated, and its subcategory of compact objects is equivalent to $D^b(\Mod(G)^\text{fp})$; see
\cite[Thm.~2.2, Cor.~2.3, p.~126]{HP23} for instance. From this perspective it is natural to ask for the graded center of $D^b(\Mod(G)^\text{fp})$. In Remark \ref{weak} below we explain 
how to identify the latter category with $D_{\Mod(G)^\text{fp}}^b(G)$, which plays a role in the categorical $p$-adic local Langlands conjectures (see \cite[Conj.~6.1.15, p.~81]{EGH25} for example). 

\begin{rems}\label{finprojdim}
By \cite[Prop.~2.10, p.~129]{HP23} one can view $D(G)^c$ as a subcategory of $D^b(\Mod(G)^\text{fp})$ (after identifying the latter with the objects in $D^b(G)$ of type 
$\text{FP}_\infty$, see \cite[Def.~2.4, Lem.~2.8]{HP23}). Furthermore, if $G$ is a $p$-adic Lie group, \cite[Prop.~2.24, p.~133]{HP23} applies since $D(G)$ is compactly generated and 
$\Mod(G)$ is AB4*-$d$ by Lemma \ref{ab4}. Thus $V^\bb \in D^b(\Mod(G)^\text{fp})$ belongs to $D(G)^c$ precisely when $V^\bb$ has 'finite projective dimension' in the sense 
that $\Hom_{D(G)}(\Sigma^{-i}V^\bb,-)=0$ for $i>\!\!>0$ (see \cite[Def.~2.22, Rem.~2.23]{HP23}).

For example, if $G$ is a pro-$p$ group, $k$ is the only irreducible object of $\Mod(G)$ and so $D^b(\Mod(G)^\text{fp})$ is the thick envelope of $k$ in $D(G)$. 
Moreover, if $G$ is a compact $p$-adic Lie group which is pro-$p$, we have the equality $D^b(\Mod(G)^\text{fp})=D(G)^c$ precisely when $H^i(G,-)=0$ for $i>\!\!>0$
(e.g. when $G$ is torsion-free). 
\end{rems}

Here is the main observation of this section.

\begin{cor}\label{loccoh}
Let $G$ be a locally pro-$p$ group and let $k$ be a field of characteristic $p$. Suppose $G$ contains no proper open centralizers and $\Mod(G)$ is locally coherent.
Then there is an isomorphism of $k$-algebras
$$
\widehat{k[Z(G)]}_\text{red} \overset{\sim}{\longrightarrow} Z^0\big(D^b(\Mod(G)^\text{fp})\big)_\text{lred}
$$
given by $Z(G) \ni c \mapsto (s \circ \Phi)(c)$.
\end{cor}

\begin{proof}
Taking $\mathcal{C}$ in Proposition \ref{DbC} to be the abelian category $\Mod(G)^\text{fp}$ we know the restriction map $\pi^b$ gives an isomorphism
$$
Z^0\big(D^b(\Mod(G)^\text{fp})\big)_\text{lred} \overset{\sim}{\longrightarrow} Z\big(\Mod(G)^\text{fp}\big)_\text{lred}
$$
with inverse given by the section $s$. It remains to identify the target with $\widehat{k[Z(G)]}_\text{red}$ via $\Phi$. Since we know (by Proposition \ref{addendum} and its subsequent paragraph) that $\Phi$ identifies the latter with 
$Z\big(\Mod(G)^\text{fg}\big)_\text{lred}$, note the switch from fp to fg, it suffices to show that the restriction map
$$
Z\big(\Mod(G)^\text{fg}\big) \longrightarrow Z\big(\Mod(G)^\text{fp}\big)
$$
induces an isomorphism on $(-)_\text{lred}$. Observe that $\Bbb{X}_{U'}$ is finitely presented for all open subgroups $U'$ since 
$(-)^{U'}$ clearly commutes with filtered colimits. Thus $\Mod(G)^\text{fp}$ contains a family of generators $\Bbb{X}_{U'}$ for $\Mod(G)$, and therefore it suffices (by Lemma \ref{serre}) to show that 
$$
\text{Serre}_{\Mod(G)}\big(\Mod(G)^\text{fp}\big)=\text{Serre}_{\Mod(G)}\big(\Mod(G)^\text{fg}\big).
$$
The inclusion $\subseteq$ is trivial as fp $\Rightarrow$ fg. To check $\supseteq$ write a finitely generated $V$ as a quotient of a finite coproduct 
$\bigoplus_{i \in I} \Bbb{X}_{U_i}\twoheadrightarrow V$. This coproduct is finitely presented, hence any quotient thereof lies in the Serre subcategory generated by 
$\Mod(G)^\text{fp}$.
\end{proof}

\begin{rem}\label{weak}
In the previous proof, note that $\Mod(G)^\text{fp}$ is a \emph{weak} Serre subcategory of $\Mod(G)$ in the sense of \cite[Def.~12.10.1]{Sta}, i.e. if 
$V_1\rightarrow V_2\rightarrow V_3 \rightarrow V_4 \rightarrow V_5$ is an exact sequence with all of $V_1, V_2, V_4, V_5$ finitely presented, then so is the middle term $V_3$. 
It may not be Serre. 

As a result it makes sense to consider the triangulated subcategory $D_{\Mod(G)^\text{fp}}^b(G)$ of complexes $V^\bb$ for which $h^i(V^\bb)$ is finitely presented for all $i\in \Z$ (see 
\cite[Sect.~13.17]{Sta} and \cite[Def.~13.2.7, p.~329]{KS06} for example, but note that weak Serre subcategories are called 'thick' in \cite[Def.~8.3.21, Rem.~8.3.22, p.~184]{KS06}, which is inconsistent with our convention in Definition \ref{thck}). 

By \cite[Thm.~13.2.8, p.~329]{KS06}, or rather the paragraph immediately after its proof, the inclusion functor induces an equivalence of categories
$$
\delta^b: D^b(\Mod(G)^\text{fp}) \overset{\sim}{\longrightarrow} D_{\Mod(G)^\text{fp}}^b(G)
$$
because the criterion there is fulfilled: For any epimorphism $W \twoheadrightarrow V$ into a $V \in \Mod(G)^\text{fp}$ there exists a morphism $V' \rightarrow W$, with 
$V' \in \Mod(G)^\text{fp}$, such that the composition $V' \rightarrow V$ is an epimorphism. (To see this pick a finite set of generators $v_i \in V$, with preimages $w_i \in W^U$ for some $U$, then this gives rise to a morphism $V':=\bigoplus_i \Bbb{X}_U \rightarrow W$ with the desired property.)
\end{rem}

We finish with an example illustrating that the computation of the full graded center, i.e. without modding out the locally nilpotents, is quite complicated even in the simplest case 
$G=\Z/2\Z$.

\begin{exmp}
We consider $G=\Z/2\Z$ and the category $\Mod(G)$ of all $k[\Z/2\Z]$-modules, where $k$ is an arbitrary field. In this example we will describe $Z^*:=Z^*\big(D^b(\Mod(G)^\text{fg})\big)$ following \cite{KY11}.

Note that $k[\Z/2\Z]\simeq k \times k$ is semisimple when $\text{char}(k)\neq 2$, so $\Ext_{\Mod(G)}^i(V,W)=0$ for $i>0$. In particular $\Mod(G)$ is hereditary, and we conclude that $Z^*=Z^0=k \times k$.

The more interesting scenario is when $\text{char}(k)=2$, in which case $k[\Z/2\Z]\simeq k[X]/(X^2)$ is the ring of dual numbers. To describe $Z^*$ we adopt the notation $T(A,M)$ from \cite[p.~453]{KY11} for a ring $A$ and an $(A,A)$-bimodule $M$. This denotes the subalgebra $A \oplus M \subset \begin{psmallmatrix} A & M \\  0 & A \end{psmallmatrix}$ consisting of matrices with identical diagonal entries. In this notation \cite[Prop.~5.4, p.~463]{KY11} gives an isomorphism of graded $k$-algebras (note that $Z^*$ is commutative, not just graded-commutative, as $\text{char}(k)=2$)
$$
Z^*\simeq T\big(k[\zeta], \prod_{r\geq 0} k\big).
$$
Here the indeterminate $\zeta$ acts as zero on $\prod_{r\geq 0} k$, and the right-hand side is graded as follows: The degree zero piece is 
$Z^0\simeq T\big(k, \prod_{r\geq 0} k\big)$ and $\text{deg}(\zeta)=1$. More concretely, as graded vector spaces we have 
$$
Z^* \simeq T\big(k[\zeta], \prod_{r\geq 0} k\big)=\big(\underbrace{k \oplus \prod_{r\geq 0} k}_{Z^0}\big) \oplus \underbrace{k\zeta}_{Z^1} \oplus \underbrace{k\zeta^2}_{Z^2} \oplus \cdots \oplus \underbrace{k\zeta^r}_{Z^r} \oplus \cdots.
$$
Of course $\widehat{k[Z(G)]}_\text{red}=Z^0_\text{lred}=k$, but $Z^0$ contains an uncountable subspace $\prod_{r\geq 0} k$ of nilpotent elements (any element hereof clearly has square zero in $Z^0$). Furthermore $\zeta \in Z^1$ is \emph{not} locally nilpotent; $\zeta$ does not act nilpotently on the trivial representation $k$. (One can see this by directly using the definition of $\zeta=\zeta_1$ in \cite[Lem.~5.3, p.~462]{KY11} and evaluating it on the indecomposable complex denoted $A_m^\infty$ in \cite[p.~461]{KY11}. Alternatively one can apply Remark \ref{Linckel} below.)
\end{exmp}

This example is complemented by a general result of Linckelmann discussed in the next remark. (We remind the reader that $H^*(\Z/2\Z,k)=k[\zeta]$ is a polynomial algebra when $\text{char}(k)=2$.)

\begin{rem}\label{Linckel}
If $G$ is an arbitrary $p$-group, and $k$ is a field of characteristic $p$, \cite[Prop.~1.3, p.~490]{Lin09} asserts that the kernel $\mathcal{N}$ of the evaluation-at-$k$ map is \emph{nilpotent} and 
there is a canonical isomorphism 
$$
Z^*\big(D^b(\Mod(G)^\text{fg})/\mathcal{N} \overset{\sim}{\longrightarrow} H^*(G,k).
$$
(From the last paragraph of Remarks \ref{finprojdim} one can easily establish an analogue of this for \emph{pro}-$p$ groups $G$. 
In the pro-$p$ case the evaluation-at-$k$ map is still a split surjection, and its kernel at least consists of \emph{locally} nilpotent elements, thereby identifying 
$Z^*\big(D^b(\Mod(G)^\text{fg})_\text{lred}$ with $H^*(G,k)_\text{nil}$.)
\end{rem}


\appendix

\section{A hypercohomology spectral sequence for derived products}

We will make use of the following general fact: Suppose $\mathcal{C}'$ is an abelian category with enough injectives, and let $F: \mathcal{C}'\rightarrow \mathcal{C}$ be a left exact functor into some abelian category $\mathcal{C}$. Assume $F$ has finite cohomological dimension.
If $W^\bb$ is any complex from $\mathcal{C}'$ there is a convergent hypercohomology spectral sequence
$$
E_2^{r,s}=R^rF\big(h^s(W^\bb)\big) \Longrightarrow h^{r+s}\big(RF(W^\bb)\big).
$$
A great reference for this is \cite[Prop.~4.4.6, p.~223]{Ver96} and the fourth convergence criterion in part b of \cite[Prop.~4.5.2, p.~224]{Ver96}.

In our application $\mathcal{C}$ is Grothendieck, and $\mathcal{C'}=\mathcal{C}^I$ for some index set $I$. Here $\mathcal{C}^I$ denotes the category of $I$-tuples. The injective objects hereof are $I$-tuples of injectives from $\mathcal{C}$. 

To say that $\mathcal{C}$ has products indexed by $I$ means the 'diagonal' functor $\Delta: \mathcal{C}\rightarrow \mathcal{C}^I$ has a right adjoint 
$\Pi: \mathcal{C}^I \rightarrow \mathcal{C}$. We let $\Pi^{(i)}:=R^i\Pi$ be its $i^\text{th}$ derived functor, which vanishes for $i>d$ if $\mathcal{C}$ is $\text{AB4}^*-d$. 
Note that $\Pi$ is left exact since it admits a left adjoint, namely $\Delta$. 

\begin{prop}\label{spectral}
Let $\mathcal{C}$ be a Grothendieck category which is $\text{AB4}^*-d$ for some $d\geq 0$. Fix an index set $I$ and suppose
$(W_n^\bb)_{n \in I}$ is a family of objects in $D(\mathcal{C})$. Then there is a convergent spectral sequence
$$
E_2^{r,s}={\prod}^{(r)} h^s(W_n^\bb) \Longrightarrow h^{r+s}\big(\prod W_n^\bb \big)
$$
where $\prod W_n^\bb$ denotes the product in $D(\mathcal{C})$ (which exists since $\mathcal{C}$ is Grothendieck, as explained in the proof). 
\end{prop}

\begin{proof}
The previous general discussion with $F=\Pi$ gives a convergent spectral sequence
$$
E_2^{r,s}={\prod}^{(r)} h^s(W_n^\bb) \Longrightarrow h^{r+s}\big(R\Pi((W_n^\bb)_{n \in I})\big).
$$
Let us unwind the target: For each $n \in I$ choose a quasi-isomorphism $W_n^\bb \rightarrow J_n^\bb$ into $K$-injective complex from $\mathcal{C}$. The latter is Grothendieck so this can be done by \cite[Thm.~19.12.6]{Sta} for example.  
This gives a quasi-isomorphism $(W_n^\bb)_{n \in I}\rightarrow (J_n^\bb)_{n \in I}$ into a $K$-injective complex from 
$\mathcal{C}^I$. To see that $(J_n^\bb)_{n \in I}$ is indeed $K$-injective, just note that the morphism complex
$$
\Hom_{\mathcal{C}^I}^\bb((A_n^\bb)_{n \in I}, (J_n^\bb)_{n \in I})=\prod_{n \in I} \Hom_\mathcal{C}^\bb(A_n^\bb, J_n^\bb)
$$
is acyclic if each $A_n^\bb$ is (since products are exact in the category of abelian groups). We conclude that
$$
R\Pi((W_n^\bb)_{n \in I})=\Pi ((J_n^\bb)_{n \in I})
$$
is the object of $D(\mathcal{C})$ represented by the complex
$$
\cdots \longrightarrow \prod J_n^{i-1} \longrightarrow \prod J_n^{i} \longrightarrow \prod J_n^{i+1} \longrightarrow \cdots.
$$
By \cite[Lem.~13.31.5]{Sta} this (is $K$-injective and) gives the product $\prod_{n \in I} J_n^\bb$ in $D(\mathcal{C})$, or which amounts to the same, the product of the objects $W_n^\bb$ in the derived category $D(\mathcal{C})$. 
\end{proof}

\subsection*{Acknowledgments} 

This research was funded by the Deutsche Forschungsgemeinschaft (DFG, German Research Foundation) -- Project-ID 427320536 -- SFB 1442, as well as under Germany’s Excellence Strategy EXC 2044/390685587, Mathematics M\"{u}nster: Dynamics–Geometry–Structure.

C. S. thanks Universit\"{a}t M\"{u}nster for its hospitality during a productive stay in June, 2026.




\bigskip

\noindent {\it{E-mail addresses}}: {\texttt{pschnei@uni-muenster.de}, {\texttt{csorensen@ucsd.edu}}

\bigskip

\noindent {\sc{Peter Schneider, Mathematics M\"{u}nster, Universit\"a{}t M\"{u{nster, M\"{u}nster, Germany.}}

\bigskip

\noindent {\sc{Claus Sorensen, Department of Mathematics, UC San Diego, La Jolla, USA.}}

\end{document}